\documentclass[11pt,a4paper]{article}

\usepackage[margin=25mm]{geometry}
\usepackage{amsmath,amssymb,amsthm}
\usepackage{enumitem}
\usepackage{microtype}
\usepackage{aliascnt}

\newtheorem{theorem}{Theorem}[section]

\newaliascnt{proposition}{theorem}
\newtheorem{proposition}[proposition]{Proposition}
\aliascntresetthe{proposition}

\newaliascnt{lemma}{theorem}
\newtheorem{lemma}[lemma]{Lemma}
\aliascntresetthe{lemma}

\newaliascnt{corollary}{theorem}
\newtheorem{corollary}[corollary]{Corollary}
\aliascntresetthe{corollary}

\theoremstyle{definition}

\newaliascnt{definition}{theorem}

\aliascntresetthe{definition}

\newaliascnt{example}{theorem}

\aliascntresetthe{example}

\theoremstyle{remark}

\newaliascnt{remark}{theorem}
\newtheorem{remark}[remark]{Remark}
\aliascntresetthe{remark}

\usepackage{hyperref}
\usepackage[nameinlink,noabbrev]{cleveref}

\hypersetup{
    colorlinks=true,
    linkcolor=blue,
    citecolor=blue,
    urlcolor=blue,
    bookmarksnumbered=true,
    pdfdisplaydoctitle=true,
    pdftitle={
        Neural network realization of binary refinement iterates
        via a two-chart atlas selector
    },
    pdfauthor={Tsogtgerel Gantumur},
    pdfsubject={Exact ReLU realization of binary refinement iterates},
    pdfkeywords={
        ReLU networks;
        refinement operators;
        cascade algorithms;
        continuous piecewise linear functions;
        exact realization;
        residual dynamics
    }
}

\newcommand{\R}{\mathbb R}
\newcommand{\Z}{\mathbb Z}
\newcommand{\T}{\mathbb T}
\newcommand{\supp}{\operatorname{supp}}
\newcommand{\ReLU}{\operatorname{ReLU}}
\newcommand{\Rep}{\operatorname{Rep}}
\newcommand{\Sw}{\operatorname{Sw}}
\newcommand{\CPwL}{\operatorname{CPwL}}
\newcommand{\Ups}{\Upsilon}
\newcommand{\emb}{\iota}
\newcommand{\mar}{\eta}
\newcommand{\trans}{\mathsf T}

\title{Neural network realization of binary refinement iterates via\\a two-chart atlas selector}
\author{Tsogtgerel Gantumur\\[0.5ex]
\small McGill University, Montr\'{e}al, QC, Canada\\
\small National University of Mongolia, Ulaanbaatar, Mongolia\\
\small Mongolian Academy of Sciences, Institute of Mathematics and Digital Technology\\[0.5ex]
\small \texttt{gantumur.tsogtgerel@mcgill.ca}}

\date{July 21, 2026}

\begin{document}
\maketitle

\begin{abstract}
Refinement operators generate many of the functions used in wavelet constructions, subdivision
schemes, and geometric modeling.  Their finite iterates can develop rapidly increasing numbers of
linear pieces, making them a natural test case for the expressive power of deep neural networks.
Earlier work showed that, for scalar binary refinement with a finitely supported mask, every
compactly supported continuous piecewise linear seed has finite refinement iterates that admit exact
ReLU realizations of fixed width and depth growing linearly with the number of refinement steps.
The present paper gives a new construction of this known realization theorem.

The main difficulty is that the binary refinement cascade is driven by discontinuous digit choices,
whereas ReLU networks produce continuous piecewise linear maps.  We resolve this difficulty by
representing the residual dynamics on a polygonal model of the circle and describing each residual
position in two overlapping coordinate systems, one ordinary and one shifted by one half.  Their
discontinuities occur at different points.  The network changes between the two descriptions only
where both are valid and the corresponding linear cascade updates already agree.  The change can
therefore be implemented exactly by a ReLU switch, without multiplying the evolving state by a
variable selector.

The same construction yields exact evaluation of every continuous piecewise linear function on the
circle satisfying the natural endpoint compatibility condition, without requiring positivity or
vanishing near the endpoints.  Localized elementary seeds are realized by a two-pass network: the
first pass computes the terminal scalar value, and the second reruns the residual trajectory while
propagating the finite-dimensional cascade state.  Translation covariance, finite decomposition into
translated elementary seeds, and gluing of local realizations then extend the construction to arbitrary
compactly supported continuous piecewise linear seeds in a preserved support window.
\end{abstract}




\section{Introduction}
\label{sec:int}

\subsection{Background and motivation}
\label{ss:intro-background}

Neural-network approximation theory seeks structural explanations for the expressive power of deep
networks and for the role of depth in representing complex functions
\cite{approx,nnapprox}.  Refinement operators provide a particularly transparent class of recursive
examples arising in approximation theory, wavelet theory, subdivision, and geometric modeling.
In the one-dimensional homogeneous binary case, such an operator has the form
\begin{equation}
\label{eq:intro-refinement-operator}
        (Vf)(t)
        =
        \sum_{j\in\Z}c_jf(2t-j),
\end{equation}
where the mask \((c_j)_{j\in\Z}\) is finitely supported.  Whereas refinable-function theory often
studies fixed points or normalized eigenfunctions of \(V\), we consider the exact ReLU realization
of the finite iterates \(V^nf\).

The scalar binary realization theorem is already known.  Daubechies et al.\ proved that the
iterates of every compactly supported CPwL seed admit exact ReLU realizations of fixed width and
depth \(O(n)\) \cite{refinable}.  The loop-controller construction of \cite{loop} subsequently gave
an exact residual-dynamics formulation for homogeneous vector-valued \(M\)-ary refinement.
Accordingly, the contribution of the present paper is not a new scalar binary conclusion, but a
different mechanism for resolving the discontinuities that enter the cascade representation.

After vectorizing the unit-interval pieces of a compactly supported function, one refinement step
becomes a finite-dimensional matrix cascade driven by the binary residual map.  Its digit choices
are discontinuous, while ReLU networks produce continuous piecewise linear functions.  Earlier
exact constructions use CPwL selector and readout mechanisms together with terminal masking:
the relevant terminal profiles vanish near residual seams, making any remaining ambiguity
harmless \cite{refinable,loop}.

Here we replace this masking mechanism inside the residual cascade by a two-chart atlas selector.
The residual circle is covered by the ordinary and half-shifted charts.  Their same-chart branch-cut
sets are disjoint, and selector transitions are confined to trusted overlaps where the competing
local descriptions are simultaneously valid.  On these overlaps, finite windows of one intrinsic
residual covector satisfy exact chart-transition identities, and the corresponding local adjoint
updates agree after passage to a common source chart.  A ReLU switch can therefore select between
them without multiplying a state by a variable selector.

The binary setting isolates this atlas mechanism in its simplest form, but the formulation is
motivated by non-product refinements and other systems in which residual dynamics naturally
requires several interacting local descriptions.  The proof keeps two roles separate: the atlas
selector controls the residual adjoint cascade, while terminal localization is used only to reduce
general CPwL seeds to translated local reference atoms.

\subsection{Main result and contribution}
\label{ss:intro-main-result}

The main theorem has the following informal form.

\begin{theorem}[Main theorem, informal form]
Let \(V\) be a scalar binary refinement operator with finitely supported mask preserving a fixed compact
support window.  If \(f:\R\to\R\) is CPwL and supported in that window, then, for every
\(n\ge1\), the function \(V^nf\) is realized exactly by a ReLU network whose width is bounded
independently of \(n\) and whose depth is \(O(n)\).
\end{theorem}

The constants may depend on the mask, the support window, and the finite CPwL complexity of
\(f\), but not on \(n\).  This is an exact real-arithmetic representation theorem; no claim is made
about finite-precision stability or bit complexity.

The proof has three principal components.  First, the exact loop controller of \cite{loop}
transports the residual orbit continuously on a polygonal model of the circle.  It also yields an
atlas scalar readout for every CPwL circle profile satisfying the natural seam compatibility
condition, without requiring positivity or endpoint vanishing.

Second, the residual cascade is formulated intrinsically as a local system over the circle.
An intrinsic residual covector is stored through finite padded windows in the ordinary and half-shifted
charts.  Fixed source--target branch matrices propagate these windows locally, and the
load-bearing finite overlap identity shows that competing adjoint updates agree on trusted chart
overlaps.

Third, reference rank-one terminal atoms supported inside one ordinary safe chart are handled by
a two-pass construction.  The first pass computes the terminal scalar factor.  The second pass
reruns the forward residual orbit while propagating the initialized atlas adjoint state.  A direct
ordinary-chart pairing then recovers the reference-atom contribution; no terminal chart selector is
needed.  Translation covariance, finite atomic decomposition, and gluing of finitely many unit-interval realizations
extend the construction to arbitrary compactly supported CPwL seeds.

This separation is the principal conceptual feature of the proof.  Atlas overlap agreement resolves
the residual cascade, whereas terminal localization serves only as a local atom-reduction device.
The trusted regions and selector transition widths are fixed independently of \(n\).  The crude
exponential gate scales that may occur arise instead from possible growth of finite adjoint products.

\subsection{Organization}
\label{ss:intro-organization}

\Cref{sec:preliminaries} develops the intrinsic residual fibers, the two lift charts, the ideal
infinite branch mechanism, and the ReLU gate for exact switching.
\Cref{sec:scalar-readout} recalls the loop controller and constructs the scalar atlas readout.
\Cref{sec:compatible-fibers} passes to padded finite covector windows and proves local adjoint
exactness and finite overlap agreement.
\Cref{sec:adjoint-branch-selectors} constructs and iterates the atlas adjoint selector.
\Cref{sec:rank-one-terminal-atoms} combines the scalar and adjoint constructions to realize
reference rank-one terminal atoms, and
\Cref{sec:general-cpwl-seeds} reduces general compactly supported CPwL seeds to these atoms.
\Cref{sec:conclusions} summarizes the result and its broader atlas-based interpretation.

\section{Preliminaries}
\label{sec:preliminaries}

This section fixes the notation and basic constructions used in the proof.  We first formulate the
refinement cascade intrinsically in terms of residual fibers and adjoint covectors.  We then introduce
the ordinary and half-shifted lift charts on the residual circle, together with their source--target
branch coordinates, and describe the ideal infinite-chart mechanism underlying overlap agreement.
We conclude with the elementary ReLU gate used to switch exactly between candidates that agree
on trusted overlaps.  The padded finite chart windows and finite branch matrices used in the actual
atlas recursion are introduced in \Cref{sec:compatible-fibers}.

\subsection{Network classes and CPwL maps}
\label{subsec:network-cpwl}

For positive integers \(W,D,d,N\), we write
\(\Ups_{W,D}(\ReLU;d,N)\)
for the set of outputs of fully connected ReLU networks with width \(W\), depth \(D\), input
dimension \(d\), and output dimension \(N\).  The final realization theorem has input dimension
\(d=1\), although several intermediate modules act on fixed finite-dimensional state spaces.

A map is called CPwL if it is continuous and piecewise affine on a locally finite polyhedral
subdivision of its domain.  For a compactly supported function on \(\R\), this is equivalent to
piecewise affinity on a finite partition of its support.

\begin{remark}
\label{rem:cpwl-basic}
We use repeatedly the following standard facts.
\begin{enumerate}[label=(\roman*),leftmargin=2em]
    \item Every scalar CPwL function on an interval with finitely many breakpoints is realized
    exactly by a one-hidden-layer ReLU network of width proportional to the number of affine
    pieces.

    \item Every fixed finite-dimensional CPwL module used below admits an exact ReLU realization
    of fixed width and depth.  Indeed, after choosing a finite polyhedral subdivision of the relevant
    compact state set, its scalar coordinate functions are continuous piecewise affine finite-element
    functions and therefore admit exact ReLU realizations \cite{relu-fem}.  Whenever such a module
    is initially defined only on a compact polyhedral state set, we fix a global CPwL extension
    before using it in a network.

    \item Affine precomposition changes neither the asymptotic width nor the depth.

    \item Finite direct sums and finite linear combinations enlarge the width only by a fixed factor.

    \item Composing a depth-\(O(n)\) network with a fixed CPwL module preserves depth \(O(n)\).
\end{enumerate}
When a function is said to be realized on \([0,1]\), this means that it agrees there with the
restriction of a globally defined ReLU-network output.
\end{remark}

Throughout the paper, ``fixed'' means independent of the refinement depth \(n\).  We do not
optimize the constants.  The weights and biases may grow with \(n\); the principal complexity
statement concerns fixed width and depth \(O(n)\).

\subsection{The refinement operator and the support window}
\label{subsec:refinement-support-window}

Let \((c_j)_{j\in\Z}\) be a finitely supported scalar mask, and define the homogeneous binary
refinement operator by
\begin{equation}
\label{eq:binary-refinement-operator}
    (Vf)(t)=\sum_{j\in\Z}c_jf(2t-j).
\end{equation}
If the mask is identically zero, then \(V^nf=0\) for every \(n\ge1\), and the main result is
immediate.  We therefore assume throughout the paper that the mask is not identically zero.

We assume that \(V\) preserves a fixed support window \([0,L]\), where \(L\ge1\) is an integer:
\begin{equation}
\label{eq:support-preservation}
    \supp f\subset[0,L]
    \qquad\Longrightarrow\qquad
    \supp Vf\subset[0,L].
\end{equation}
This hypothesis reduces the cascade to a fixed finite-dimensional state space.  The constants in
the realization bounds may depend on \(L\) and on the finite mask support, but are independent of
the refinement depth \(n\).

\subsection{Residual fibers and the invariant cascade}
\label{subsec:residual-fibers-invariant-cascade}

Vectorization is natural because all arguments appearing in one refinement step have the same
residual class modulo integers.  We therefore separate the residual dynamics from the coordinate
sequence indexed by integer translates.
Let
\[
        \T:=\R/\Z,
        \qquad
        d:\T\to\T,
        \qquad
        d([y]):=[2y],
\]
be the residual circle and its doubling map.  The algebraic model fiber is
\(E:=\R^{(\Z)}\), the space of finitely supported bi-infinite sequences.  For
\(\kappa\in\Z\), define the shift \(C_\kappa:E\to E\) by
\begin{equation}
\label{eq:c-shift}
        (C_\kappa w)_i:=w_{i+\kappa}.
\end{equation}
Then we have \(C_\kappa^{-1}=C_{-\kappa}\).
Now define the {\em residual local system} by
\[
        \mathcal E
        :=
        (\R\times E)/{\sim},
        \qquad
        (\tilde y,w)
        \sim
        (\tilde y+\kappa,C_\kappa w),
        \quad
        \kappa\in\Z.
\]
Its projection is \([\tilde y,w]\mapsto[\tilde y]\in\T\), and its fiber over \(y\) is denoted by \(\mathcal E_y\).
In geometric language, \(\mathcal E\) is a flat local system over \(\T\), or equivalently the
algebraic analogue of a flat vector bundle with fiber \(E\) and monodromy given by the shift
\(C_1\).

For a fixed \(y\in\T\), a lift of \(y\) is a number \(\tilde y\in\R\) satisfying
\([\tilde y]=y\).  Such a choice identifies \(\mathcal E_y\) with \(E\): the vector
\(w\in E\) represents the intrinsic element \([\tilde y,w]\in\mathcal E_y\).  If the lift is
changed from \(\tilde y\) to \(\tilde y+\kappa\), then the same intrinsic element is represented
by \(C_\kappa w\).

More generally, a \emph{lift chart} consists of a connected open set \(U\subset\T\) and a continuous map
\[
        \ell:U\to\R,
        \qquad
        [\ell(y)]=y
        \quad
        \text{for every }y\in U.
\]
Its image \(\widetilde U:=\ell(U)\) is an open interval on which the covering projection
\(\pi:\R\to\T\) restricts to a homeomorphism
\(\pi|_{\widetilde U}:\widetilde U\longrightarrow U\),
with inverse \(\ell\).  Thus \(\tilde y=\ell(y)\) is the real coordinate of \(y\) in this lift
chart.  All chart coordinates used below are of this form.

A compactly supported function \(f:\R\to\R\) determines a section \(W_f\) of
\(\mathcal E\).  In the lift chart \((U,\ell)\), its coordinate representative is the map
\[
        w_f^\ell:U\to E,
        \qquad
        w_f^\ell(y)
        :=
        \bigl(f(\ell(y)+i)\bigr)_{i\in\Z}.
\]
Equivalently, as a function of the numerical lift coordinate
\(\tilde y\in\widetilde U\), one may write
\[
        \widetilde w_f^\ell(\tilde y)
        :=
        \bigl(f(\tilde y+i)\bigr)_{i\in\Z},
        \qquad
        w_f^\ell
        =
        \widetilde w_f^\ell\circ\ell.
\]
If \((U,\ell)\) and \((U',\ell')\) are two lift charts, then on each connected component of
\(U\cap U'\) there is an integer \(\kappa\) such that
\(\ell'(y)=\ell(y)+\kappa\).
On that component, their coordinate representatives satisfy
\[
        w_f^{\ell'}(y)
        =
        C_\kappa w_f^\ell(y).
\]
Hence the local representatives patch together to define one intrinsic section
\(W_f(y)\in\mathcal E_y\), \(y\in\T\).

The refinement operator induces an intrinsic linear map
\[
        \mathcal B_y:\mathcal E_{d(y)}\longrightarrow\mathcal E_y.
\]
To describe it in coordinates, choose a target lift \(\tilde y\) of \(y\) and a source lift
\(\tilde y^+\) of \(d(y)\).  They determine an integer \(q\) through
\begin{equation}
\label{eq:invariant-lift-relation}
        2\tilde y=q+\tilde y^+.
\end{equation}
Here the fiber over \(d(y)\) is the source fiber and the fiber over \(y\) is the target fiber.
The integer \(q\) is the source--target branch label in these coordinates.

For \(w\in E\), define the infinite branch matrix by
\[
        (T_q^\infty w)_i
        :=
        \sum_{\ell\in\Z}
        c_{q+2i-\ell}w_\ell.
\]
The mask is finitely supported, so \(T_q^\infty:E\to E\) is well defined.  The intrinsic map is
given in the chosen lifts by
\[
        \mathcal B_y[\tilde y^+,w]
        :=
        [\tilde y,T_q^\infty w].
\]
This definition is independent of the lifts.  Indeed, if the target and source lifts are changed
by \(a,b\in\Z\), respectively, then the new branch label is \(q'=q+2a-b\), and a direct
coordinate calculation gives
\begin{equation}
\label{eq:T-transform}
        T_{q'}^\infty C_b
        =
        C_aT_q^\infty.
\end{equation}
Thus the resulting representatives determine the same intrinsic target element.

For a compactly supported function \(f\), the coordinate calculation
\[
        (Vf)(\tilde y+i)
        =
        \sum_{j\in\Z}c_jf(2\tilde y+2i-j)
        =
        \sum_{\ell\in\Z}
        c_{q+2i-\ell}f(\tilde y^++\ell)
\]
shows that
\[
        W_{Vf}(y)
        =
        \mathcal B_yW_f(d(y)).
\]
More generally, starting from any section \(W_0\) of \(\mathcal E\), define its primal cascade by
\begin{equation}
\label{eq:invariant-primal-cascade}
        W_{m+1}(y)
        :=
        \mathcal B_yW_m(d(y)),
        \qquad
        m\ge0.
\end{equation}
For the section induced by \(f\), this gives \(W_m=W_{V^mf}\).

To represent point observables on residual sections, we introduce the dual residual fibers.
Since \(E=\R^{(\Z)}\), its algebraic dual is \(E^*\cong\R^\Z\), with pairing
\[
        \langle\psi,w\rangle
        :=
        \sum_{i\in\Z}\psi_iw_i.
\]
The sum is finite because \(w\) is finitely supported.  The dual residual local system is
\[
        \mathcal E^*
        :=
        (\R\times E^*)/{\sim},
        \qquad
        (\tilde y,\psi)
        \sim
        (\tilde y+\kappa,C_{-\kappa}^*\psi),
\]
where the algebraic dual \(C_\kappa^*:E^*\to E^*\) is defined by
\(\langle C_\kappa^*\psi,w\rangle=\langle\psi,C_\kappa w\rangle\),
or equivalently,
\[
        (C_\kappa^*\psi)_i
        =
        \psi_{i-\kappa}.
\]

An element \(\Phi_y\in\mathcal E_y^*\) acts on the value \(W(y)\in\mathcal E_y\) of a residual
section and thereby defines the intrinsic point observable
\(W\mapsto\langle\Phi_y,W(y)\rangle\).
Indeed, if \(\psi^{\tilde y}\) represents \(\Phi_y\) in the lift \(\tilde y\), then its
representative in the lift \(\tilde y+\kappa\) is
\begin{equation}
\label{eq:covector-transform}
        \psi^{\tilde y+\kappa}
        =
        C_{-\kappa}^*\psi^{\tilde y}.
\end{equation}
Together with the primal lift-change rule, this gives
\(\langle C_{-\kappa}^*\psi, C_\kappa w \rangle = \langle\psi,w\rangle\),
so the observable is independent of the chosen lift.

The dual of the intrinsic refinement map is
\[
        \mathcal B_y^*:
        \mathcal E_y^*
        \longrightarrow
        \mathcal E_{d(y)}^*.
\]
In the target lift \(\tilde y\) and source lift \(\tilde y^+\), it is represented by
\((T_q^\infty)^*:E^*\to E^*\), where
\[
        \bigl\langle
                (T_q^\infty)^*\psi,w
        \bigr\rangle
        =
        \bigl\langle
                \psi,T_q^\infty w
        \bigr\rangle.
\]
Its coordinate formula is
\[
        \bigl((T_q^\infty)^*\psi\bigr)_\ell
        =
        \sum_{i\in\Z}
        c_{q+2i-\ell}\psi_i.
\]
For each \(\ell\), only finitely many indices \(i\) contribute.

Let \(y_m:=d^m(y_0)\), and choose an initial intrinsic covector
\(\Phi_0\in\mathcal E_{y_0}^*\).  Its adjoint cascade is defined by
\begin{equation}
\label{eq:invariant-adjoint-cascade}
        \Phi_{m+1}
        :=
        \mathcal B_{y_m}^*\Phi_m,
        \qquad
        m=0,\ldots,n-1.
\end{equation}
For the primal cascade \eqref{eq:invariant-primal-cascade}, duality gives
\begin{equation}
\label{eq:invariant-pairing-preservation}
        \bigl\langle
                \Phi_m,W_{n-m}(y_m)
        \bigr\rangle
        =
        \bigl\langle
                \Phi_{m+1},W_{n-m-1}(y_{m+1})
        \bigr\rangle.
\end{equation}
Iterating this identity yields
\begin{equation}
\label{eq:invariant-terminal-pairing}
        \bigl\langle
                \Phi_0,W_n(y_0)
        \bigr\rangle
        =
        \bigl\langle
                \Phi_n,W_0(y_n)
        \bigr\rangle.
\end{equation}
Thus a pairing of the initial covector with the \(n\)-th evolved primal section at \(y_0\)
can be evaluated by propagating the covector forward along the residual orbit and pairing the
resulting covector \(\Phi_n\) with the initial primal section evaluated at the terminal residual
point \(y_n\).  This is the adjoint formulation used below.

\subsection{Charts, lifts, and source--target branch coordinates}
\label{subsec:residual-circle-charts}

We reserve \(t\) for the external scalar input of the realizing network and \(x\) for local chart
coordinates.  A generic representative of a point of \(\T\) will be denoted by \(\tau\).

We now fix the two lift charts used by the atlas selector.  Set
\[
        \mathcal A:=\{0,s\},
        \qquad\text{where}\quad
        s:=\textstyle\frac12.
\]
The labels \(0\) and \(s\) refer to the ordinary and half-shifted charts, respectively.
For \(\beta\in\mathcal A\), let
\[
        U_\beta:=\T\setminus\{[\beta]\}.
\]
Every \(y\in U_\beta\) has a unique coordinate \(x_\beta(y)\in(0,1)\) satisfying
\(y=[\beta+x_\beta(y)]\).
Thus the chart chooses the lift
\[
        \tilde y_\beta(y):=\beta+x_\beta(y).
\]
The excluded point \([\beta]\) is the chart seam.  Since the two seams are distinct, the domains
\(U_0\) and \(U_s\) cover \(\T\).
Using the representative \(y=[\tau]\) with \(\tau\in[0,1)\), we can write the coordinates as
\[
        x_0([\tau])=\tau,
        \qquad
        0<\tau<1,
\]
and
\[
        x_s([\tau])
        =
        \begin{cases}
        \tau-s,   & s<\tau<1,\\
        \tau-s+1, & 0\le \tau<s.
        \end{cases}
\]
The point \(\tau=s\) is excluded from the shifted chart.

Although the genuine coordinate interval is \((0,1)\), we shall also use the canonical closed
parametrization
\[
        x\longmapsto[\beta+x],
        \qquad
        0\le x\le1.
\]
Its two endpoints represent the same seam point \([\beta]\).  This closed coordinate window is
used only for CPwL extensions, one-sided branch formulas, and endpoint compatibility.  Statements
of chart exactness will be made on safe subsets of the corresponding open chart domain, except
when seam values are treated explicitly.

Fix a target chart \(\beta\in\mathcal A\) over \(y\) and a source chart
\(\alpha\in\mathcal A\) over \(d(y)\).  
This terminology agrees with the direction
\(\mathcal B_y:\mathcal E_{d(y)}\rightarrow\mathcal E_y\).
In the ordinary one-chart description, the binary residual
map is written as \(2x=q+R(x)\), where \(q\in\{0,1\}\) is the binary digit.  With two lift
charts, the same integer bookkeeping also includes the chart offsets.  
Suppose that \(y\in U_\beta\) and \(d(y)\in U_\alpha\).  
As the corresponding lifts are
\(\tilde y=x_\beta(y)+\beta\) and \(\tilde y^+=x_\alpha(d(y))+\alpha\),
the relation
\eqref{eq:invariant-lift-relation} determines a unique integer \(q\in\Z\) such that
\begin{equation}
\label{eq:source-target-branch-relation}
        2\bigl(x_\beta(y)+\beta\bigr)
        =
        q+\alpha+x_\alpha(d(y)).
\end{equation}
Thus \(q\) is the chart-dependent analogue of the usual binary digit: it is the integer
translation needed to express the doubled target lift in the chosen source chart.

For fixed \(\alpha,\beta,q\), the residual map has the affine coordinate expression
\begin{equation}
\label{eq:source-target-residual}
        x_\alpha(d(y))
        =
        R_{\alpha,\beta,q}\bigl(x_\beta(y)\bigr),
        \qquad
        R_{\alpha,\beta,q}(x)
        :=
        2(x+\beta)-\alpha-q.
\end{equation}
We call this a {\em local branch} of the residual map.  Here the chart pair
\((\beta,\alpha)\) specifies the target and source coordinates, while \(q\) specifies which
integer branch of the lifted doubling relation is used.  Its open domain of validity is
\begin{equation}
\label{eq:source-target-validity}
        \Omega_{\alpha,\beta,q}
        :=
        \left\{
                x\in(0,1):
                0<R_{\alpha,\beta,q}(x)<1
        \right\}.
\end{equation}
If \(x\in\Omega_{\alpha,\beta,q}\), then we have
\[
        d([\beta+x])
        =
        [\alpha+R_{\alpha,\beta,q}(x)],
\]
and both the target coordinate \(x\) and the source coordinate
\(R_{\alpha,\beta,q}(x)\) lie in the corresponding open chart intervals.

The branch domain is determined by the two coordinate conditions
\(0<x<1\) and \(0<R_{\alpha,\beta,q}(x)<1\).  Its endpoints can therefore arise either because the
source coordinate reaches \(0\) or \(1\), equivalently
\(2x+2\beta-\alpha\in\Z\), or because the target coordinate reaches \(x=0\) or \(x=1\).
In the first case, the image residual reaches the seam of the chosen source chart and the integer
label \(q\) changes across the point.  In the second case, the chosen target coordinate reaches one
of the two representatives of its chart seam, while the source coordinate may either remain in the
interior of its chart or simultaneously reach its seam.  These distinctions are relative to the
chosen lift charts rather than intrinsic features of \(\T\).  On the closure of each nonempty branch
domain, we retain the same affine formula \(R_{\alpha,\beta,q}\) as a one-sided coordinate formula.
The resulting local branches are listed in
\Cref{tab:one-dimensional-half-shift-branches}.

\begin{table}[t]
\centering
\begingroup
\renewcommand{\arraystretch}{1.35}
\begin{tabular}{c c c c c c c}
\hline
\(\alpha\) & \(\beta\)
& \(R_{\alpha,\beta,q}(x)\)
& \(q=-1\) & \(q=0\) & \(q=1\) & \(q=2\) \\
\hline
\(0\) & \(0\)
&
\(2x-q\)
&
\(\varnothing\)
&
\((0,\frac12)\)
&
\((\frac12,1)\)
&
\(\varnothing\)
\\[0.6em]

\(s\) & \(0\)
&
\(2x-\frac12-q\)
&
\((0,\frac14)\)
&
\((\frac14,\frac34)\)
&
\((\frac34,1)\)
&
\(\varnothing\)
\\[0.6em]

\(0\) & \(s\)
&
\(2x+1-q\)
&
\(\varnothing\)
&
\(\varnothing\)
&
\((0,\frac12)\)
&
\((\frac12,1)\)
\\[0.6em]

\(s\) & \(s\)
&
\(2x+\frac12-q\)
&
\(\varnothing\)
&
\((0,\frac14)\)
&
\((\frac14,\frac34)\)
&
\((\frac34,1)\)
\\
\hline
\end{tabular}
\endgroup
\caption{The entry in the column \(q\) is the branch-validity interval
\(\Omega_{\alpha,\beta,q}\).}
\label{tab:one-dimensional-half-shift-branches}
\end{table}

The {\em branch cuts} for the chart pair \((\beta,\alpha)\) are the target-coordinate points at which
the source coordinate reaches a seam and the integer label \(q\) changes.  Away from these points,
the label is
\begin{equation}
\label{eq:source-target-label}
        Q_{\alpha,\beta}(x)
        :=
        \lfloor 2x+2\beta-\alpha\rfloor,
\end{equation}
and the corresponding residual coordinate is
\[
        R_{\alpha,\beta}(x)
        =
        2(x+\beta)-\alpha-Q_{\alpha,\beta}(x).
\]
The construction developed in this paper does not evaluate the discontinuous map \(Q_{\alpha,\beta}\) globally.  
Instead, it uses the finitely many affine branch formulas and switches between their outputs only where the relevant candidates agree.

For the same-chart branches \(\alpha=\beta\), we have
\[
        R_{\alpha,\alpha,q}(x)
        =
        2x+\alpha-q.
\]
The ordinary chart has a branch cut at \(x=\tfrac12\), whereas the half-shifted chart has branch
cuts at \(x=\tfrac14\) and \(x=\tfrac34\).  Since these cut sets are disjoint, one chart remains
regular wherever the other changes branch, which is the basic geometric reason the two-chart
selector can avoid a global discontinuous digit choice.

\subsection{Infinite chart representatives and branch identities}
\label{subsec:ideal-infinite-chart-mechanism}

We now express the intrinsic primal and adjoint maps in the two lift charts.  The construction in
this subsection uses full infinite coordinate representatives.  The network will later retain only
their finite windows.

Let \(y\in U_\beta\).  For \(W_y\in\mathcal E_y\), denote by
\[
        w_\beta^\infty(y,W_y)\in E
\]
its representative in the chart lift
\(\tilde y_\beta(y)=\beta+x_\beta(y)\).  Similarly, for
\(\Phi_y\in\mathcal E_y^*\), denote its representative in the same lift by
\[
        \psi_\beta^\infty(y,\Phi_y)\in E^*.
\]
If \(y\in U_\beta\cap U_{\beta'}\) and the two lifts satisfy
\[
        \tilde y_{\beta'}(y)
        =
        \tilde y_\beta(y)+\kappa,
        \qquad
        \kappa\in\Z,
\]
then the lift-change rules from \Cref{subsec:residual-fibers-invariant-cascade} give
\begin{equation}
\label{eq:infinite-chart-lift-change}
        w_{\beta'}^\infty(y,W_y)
        =
        C_\kappa w_\beta^\infty(y,W_y),
        \qquad\text{and}\qquad
        \psi_{\beta'}^\infty(y,\Phi_y)
        =
        C_{-\kappa}^*
        \psi_\beta^\infty(y,\Phi_y).
\end{equation}
The two transformations preserve the pairing.

For a compactly supported function \(f:\R\to\R\) and a chart offset
\(\beta\in\mathcal A\), define the extended coordinate profile
\[
        \mathcal G_{\beta,f}(x)
        :=
        \bigl(f(x+\beta+i)\bigr)_{i\in\Z},
        \qquad
        x\in[0,1].
\]
For \(0<x<1\), this agrees with the coordinate representative of the intrinsic section:
\[
        \mathcal G_{\beta,f}(x)
        =
        w_\beta^\infty\bigl([\beta+x],W_f([\beta+x])\bigr).
\]
The values at \(x=0\) and \(x=1\) are the two one-sided coordinate representatives associated
with the seam of the \(\beta\)-chart.

For \(\alpha,\beta\in\mathcal A\) and \(q\in\Z\), write
\(T_{\alpha,\beta,q}^\infty:=T_q^\infty\)
so that
\begin{equation}
\label{eq:infinite-branch-matrix}
        \bigl(T_{\alpha,\beta,q}^\infty w\bigr)_i
        =
        \sum_{\ell\in\Z}
        c_{q+2i-\ell}w_\ell.
\end{equation}
The entries depend only on \(q\); the additional indices record the source and target charts in
which the branch is used.
Let \(y\in U_\beta\), set \(y^+:=d(y)\), and suppose that \(y^+\in U_\alpha\).  
Write
\(x:=x_\beta(y)\) and \(r:=x_\alpha(y^+)\),
and let \(q\) be determined by
\[
        2(x+\beta)=q+\alpha+r.
\]
For every \(W^+\in\mathcal E_{y^+}\), the coordinate form of the intrinsic primal map is
\begin{equation}
\label{eq:infinite-primal-branch-identity}
        w_\beta^\infty
        \bigl(y,\mathcal B_yW^+\bigr)
        =
        T_{\alpha,\beta,q}^\infty
        w_\alpha^\infty(y^+,W^+).
\end{equation}
In particular, for a compactly supported function \(f\), one obtains
\begin{equation}
\label{eq:infinite-chart-cascade-identity}
        \mathcal G_{\beta,Vf}(x)
        =
        T_{\alpha,\beta,q}^\infty
        \mathcal G_{\alpha,f}(r).
\end{equation}
The same formula remains valid at a branch endpoint as an identity between the corresponding
one-sided closed-window representatives.

Indeed, the source--target relation gives
\[
        (Vf)(x+\beta+i)
        =
        \sum_{j\in\Z}
        c_j f(2x+2\beta+2i-j)
        =
        \sum_{j\in\Z}
        c_j f(r+\alpha+q+2i-j).
\]
The substitution \(\ell=q+2i-j\) then yields \eqref{eq:infinite-chart-cascade-identity}.

The corresponding adjoint identity follows by duality.  Let
\(\Phi_y\in\mathcal E_y^*\), and let
\(\Phi_{y^+}:=\mathcal B_y^*\Phi_y\in\mathcal E_{y^+}^*\).
Then its source-chart representative is
\begin{equation}
\label{eq:infinite-adjoint-branch-identity}
        \psi_\alpha^\infty(y^+,\Phi_{y^+})
        =
        \bigl(T_{\alpha,\beta,q}^\infty\bigr)^*
        \psi_\beta^\infty(y,\Phi_y).
\end{equation}
In coordinates, the dual branch map is given by
\begin{equation}
\label{eq:infinite-adjoint-branch-identity-coord}
        \left(
                \bigl(T_{\alpha,\beta,q}^\infty\bigr)^*
                \psi
        \right)_\ell
        =
        \sum_{i\in\Z}
        c_{q+2i-\ell}\psi_i.
\end{equation}

We next record the overlap identity underlying the atlas selector.  Suppose that
\(\beta,\beta'\in\mathcal A\) are both valid target charts at \(y\), and that their lifts satisfy
\[
        \tilde y_{\beta'}(y)
        =
        \tilde y_\beta(y)+\kappa.
\]
Fix a source chart \(\alpha\) at \(y^+=d(y)\), and use the corresponding source lift
\(\tilde y^+ = \alpha+x_\alpha(y^+)\).
If the branch labels associated with the target lifts \(\beta\) and \(\beta'\) are \(q\) and
\(q'\), respectively, then
\[
        q'=q+2\kappa.
\]
The coordinate-change identity \eqref{eq:T-transform} from
\Cref{subsec:residual-fibers-invariant-cascade}, with the source chart and its lift fixed, therefore
becomes
\begin{equation}
\label{eq:infinite-primal-branch-transition}
        T_{\alpha,\beta',q'}^\infty
        =
        C_\kappa T_{\alpha,\beta,q}^\infty.
\end{equation}
Now let \(\Phi_y\in\mathcal E_y^*\), and abbreviate its target-chart representatives by
\[
        \psi_\beta
        :=
        \psi_\beta^\infty(y,\Phi_y),
        \qquad
        \psi_{\beta'}
        :=
        \psi_{\beta'}^\infty(y,\Phi_y).
\]
By \eqref{eq:infinite-chart-lift-change}, one has
\(\psi_{\beta'}=C_{-\kappa}^*\psi_\beta\).  Taking duals in
\eqref{eq:infinite-primal-branch-transition} therefore gives
\begin{equation}
\label{eq:infinite-overlap-agreement}
        \bigl(T_{\alpha,\beta,q}^\infty\bigr)^*
        \psi_\beta
        =
        \bigl(T_{\alpha,\beta',q'}^\infty\bigr)^*
        \psi_{\beta'}.
\end{equation}
Thus two compatible target-chart descriptions produce exactly the same adjoint update when both
are written in the same source chart.

In an ideal infinite atlas step, one could form, for each source chart \(\alpha\), the local
candidates
\[
        Y_{\alpha,\beta}
        :=
        \bigl(T_{\alpha,\beta,q_{\alpha,\beta}}^\infty\bigr)^*
        \psi_\beta,
        \qquad
        \beta\in\mathcal A.
\]
Equation \eqref{eq:infinite-overlap-agreement} shows that all simultaneously valid candidates agree.
Two obstacles prevent a direct network implementation of this infinite construction.  First, the
valid chart and branch labels change discontinuously with the residual point.  Second, the network
can store only finitely many coordinates of each chart representative.  
We address these issues in
\S\ref{sec:compatible-fibers}--\S\ref{sec:adjoint-branch-selectors}  
by introducing padded finite chart windows and CPwL switches supported on trusted agreement regions.

\subsection{The ReLU gate and exact switching}
\label{subsec:relu-gate-switch}

We use one elementary gate throughout the scalar and adjoint atlas constructions.  For \(A>0\),
\(\lambda\in[0,1]\), and \(z\in\R\), define
\begin{equation}
\label{eq:relu-gate}
        \Pi_A(\lambda,z)
        :=
        \ReLU\bigl(z-A(1-\lambda)\bigr)
        -
        \ReLU\bigl(-z-A(1-\lambda)\bigr).
\end{equation}
For vector inputs, the gate is applied componentwise.

\begin{lemma}[Gate identities and contraction]
\label{lem:gate-identities}
For \(A>0\) and \(\lambda\in[0,1]\), the gate satisfies
\[
        \Pi_A(1,z)=z,
        \qquad
        \Pi_A(\lambda,0)=0,
        \qquad
        z\in\R.
\]
Moreover, we have
\begin{equation}
\label{eq:gate-contraction}
        |\Pi_A(\lambda,z)|
        \le |z|,
        \qquad
        \lambda\in[0,1],\quad z\in\R.
\end{equation}
If \(|z|\le A\), then also \(\Pi_A(0,z)=0\).
\end{lemma}

\begin{proof}
The first two identities follow directly from
\(\ReLU(z)-\ReLU(-z)=z\).  Writing
\(\tau:=A(1-\lambda)\ge0\), one has
\[
        \Pi_A(\lambda,z)
        =
        \begin{cases}
        z-\tau,  & z>\tau,\\
        0,       & |z|\le\tau,\\
        z+\tau,  & z<-\tau.
        \end{cases}
\]
This gives \eqref{eq:gate-contraction}.  When \(\lambda=0\), one has \(\tau=A\), so the gate
vanishes for \(|z|\le A\).
\end{proof}

For \(Y_0,Y_1\in\R^N\), \(\theta\in[0,1]\), and \(A>0\), define
\begin{equation}
\label{eq:exact-switch}
        \operatorname{Sw}_A(\theta;Y_0,Y_1)
        :=
        Y_0+
        \Pi_A\bigl(1-\theta,Y_1-Y_0\bigr),
\end{equation}
where the gate is applied componentwise.

\begin{lemma}[Exact switching on agreement regions]
\label{lem:exact-switch}
Assume that
\[
        \|Y_1-Y_0\|_\infty\le A
\]
and that \(Y_0=Y_1\) whenever \(0<\theta<1\).  Then
\[
        \operatorname{Sw}_A(\theta;Y_0,Y_1)
        =
        \begin{cases}
        Y_0,     & \theta=1,\\
        Y_1,     & \theta=0,\\
        Y_0=Y_1, & 0<\theta<1.
        \end{cases}
\]
\end{lemma}

\begin{proof}
For \(\theta=1\), the scale assumption and
\(\Pi_A(0,Y_1-Y_0)=0\) give the value \(Y_0\).  For \(\theta=0\), the identity
\(\Pi_A(1,z)=z\) gives \(Y_1\).  If \(0<\theta<1\), then \(Y_1-Y_0=0\), and the switch returns
their common value.
\end{proof}

\begin{remark}
\label{rem:switch-not-multiplication}
The switch does not implement multiplication by the variable factor \(\theta\).  Its exactness
comes from restricting fractional switching to regions where the candidate values already agree.
The scale \(A\) is required only to control their difference on the states reached by the
construction.
\end{remark}

\section{Scalar atlas readout}
\label{sec:scalar-readout}

This section constructs an exact scalar readout on the residual circle.  The residual orbit is
transported exactly on a fixed polygonal loop, and the value of a CPwL circle profile \(h\) is
recovered from the ordinary and half-shifted chart coordinates by switching only where the two
local readouts agree.  Applied to \(y=d^n([t])\), the construction produces the terminal factor
\(h(d^n([t]))\).  It requires only the circle compatibility condition \(h(0)=h(1)\), with no
positivity or seam-vanishing assumption.

\subsection{A polygonal loop model for the residual circle}
\label{subsec:residual-circle-loop-model}

Fix a simple polygonal loop \(\Gamma\subset\R^2\) and a CPwL parametrization
\(\emb:[0,1]\to\Gamma\) such that \(\emb(0)=\emb(1)\) and \(\emb\) is injective on
\([0,1)\).  
It descends to a
homeomorphism
\[
        \emb:\T\longrightarrow\Gamma.
\]
The following result was established in \cite{loop} and is recalled here for convenience.

\begin{lemma}[Exact loop controller]
\label{lem:loop-controller}
There exists a fixed CPwL map \(D:\R^2\to\R^2\) such that
\begin{equation}
\label{eq:loop-controller-intertwining}
        D(\emb(y))
        =
        \emb(d(y)),
        \qquad
        y\in\T.
\end{equation}
Consequently, if
\[
        z_0(t):=\emb([t]),
        \qquad
        z_{m+1}(t):=D(z_m(t)),
\]
then
\begin{equation}
\label{eq:exact-loop-orbit}
        z_m(t)
        =
        \emb(d^m([t]))
        =
        \emb([2^mt]),
        \qquad
        m\ge0.
\end{equation}
For every \(n\ge1\), the map \(t\mapsto z_n(t)\) is realized on \([0,1]\) by a ReLU network
of fixed width and depth \(O(n)\).
\end{lemma}

The loop model and the controller \(D\) depend only on the binary residual dynamics and the fixed
polygonal parametrization.  They are independent of the refinement mask and of the seed.

\subsection{Scalar chart profiles}
\label{subsec:scalar-chart-profiles}

Let \(h:[0,1]\to\R\) be CPwL and satisfy \(h(0)=h(1)\).
Then \(h\) descends to a well-defined CPwL function on \(\T\), which we denote by the same symbol.

For each \(\beta\in\mathcal A=\{0,s\}\), define the extended chart profile
\begin{equation}
\label{eq:scalar-chart-profile}
        h_\beta(x)
        :=
        h([\beta+x]),
        \qquad
        0\le x\le1.
\end{equation}
For \(0<x<1\), this is the expression of \(h\) in the \(\beta\)-chart.  The endpoint values
correspond to the two one-sided values at the chart seam and agree because
\(h\) is a well-defined function on \(\T\).
More concretely, the ordinary profile is \(h_0(x)=h(x)\), and the shifted profile at
\(s=\frac12\) is
\[
        h_s(x)
        =
        \begin{cases}
        h(s+x),   & 0\le x\le1-s,\\
        h(s+x-1), & 1-s\le x\le1.
        \end{cases}
\]
At the joining point \(x=1-s\), the two expressions give \(h(1)\) and \(h(0)\), respectively, so
they agree.  Hence \(h_0\) and \(h_s\) are CPwL on \([0,1]\).  Moreover,
\[
        h_\beta(0)=h_\beta(1),
        \qquad
        \beta\in\mathcal A,
\]
because the two endpoints represent the same seam point of the corresponding chart.

\subsection{Safe arcs and local chart readouts}
\label{subsec:chart-readouts-from-loop}

Fix once and for all a margin parameter
\(0<\mar<\frac14\).
For each \(\beta\in\mathcal A\), define the {\em safe arc}
\begin{equation}
\label{eq:safe-chart-arc}
        K_\beta
        :=
        \left\{
                y\in U_\beta:
                \mar\le x_\beta(y)\le1-\mar
        \right\}.
\end{equation}
Thus \(K_\beta\) is the closed arc obtained by removing a margin \(\mar\) on both sides of the
\(\beta\)-chart seam.  Since \(s=\frac12\) and \(\mar<\frac14\), the two safe arcs cover the circle:
\begin{equation}
\label{eq:safe-arcs-cover}
        K_0\cup K_s=\T.
\end{equation}

For each \(\beta\in\mathcal A\), choose a fixed global CPwL map
\(\rho_\beta:\R^2\longrightarrow[0,1]\)
such that
\begin{equation}
\label{eq:chart-coordinate-readout}
        \rho_\beta(\emb(y))
        =
        x_\beta(y),
        \qquad
        y\in K_\beta.
\end{equation}
Such a map exists because \(x_\beta\circ\emb^{-1}\) is CPwL on the polygonal arc
\(\emb(K_\beta)\).  After a finite subdivision of \(\Gamma\), it may be extended affinely over a
finite triangulation and then fixed as a global CPwL map.  Its values away from
\(\emb(K_\beta)\) are immaterial.

Define the two scalar candidates by
\begin{equation}
\label{eq:scalar-chart-candidates}
        H_\beta(z)
        :=
        h_\beta\bigl(\rho_\beta(z)\bigr),
        \qquad
        \beta\in\mathcal A.
\end{equation}
They are global CPwL functions.  On the corresponding safe arcs, they satisfy
\begin{equation}
\label{eq:Hbeta-exact-on-Kbeta}
        y\in K_\beta
        \quad\Longrightarrow\quad
        H_\beta(\emb(y))
        =
        h_\beta(x_\beta(y))
        =
        h(y).
\end{equation}
In particular, we have
\begin{equation}
\label{eq:scalar-candidate-overlap-agreement}
        H_0(\emb(y))
        =
        H_s(\emb(y))
        =
        h(y),
        \qquad
        y\in K_0\cap K_s.
\end{equation}

\subsection{The scalar atlas selector}
\label{subsec:scalar-switch}

We now combine the two chart candidates by switching only on their agreement region.  Choose a
fixed global CPwL map
\(\vartheta:\R^2\longrightarrow[0,1]\)
such that, for \(y\in\T\),
\begin{equation}
\label{eq:scalar-selector-properties}
\begin{aligned}
        \vartheta(\emb(y))=1
        &\quad\Longrightarrow\quad y\in K_0,\\
        \vartheta(\emb(y))=0
        &\quad\Longrightarrow\quad y\in K_s,\\
        0<\vartheta(\emb(y))<1
        &\quad\Longrightarrow\quad y\in K_0\cap K_s.
\end{aligned}
\end{equation}
Equivalently, \(\vartheta=1\) on \(K_0\setminus K_s\), \(\vartheta=0\) on
\(K_s\setminus K_0\), and all transition regions lie inside \(K_0\cap K_s\).

Such a selector exists because \(K_0\) and \(K_s\) are closed polygonal arcs covering \(\T\),
with nonempty overlap.  One first defines a CPwL transition on the loop, with its fractional bands
contained in the two overlap components, and then fixes a global CPwL extension to \(\R^2\).

Set
\begin{equation}
\label{eq:scalar-readout-gate-scale}
        A_h
        :=
        \max\Big\{
                1,\,
                \max_{z\in\Gamma}
                |H_s(z)-H_0(z)|
        \Big\},
\end{equation}
and define
\begin{equation}
\label{eq:scalar-two-grid-readout}
        H(z)
        :=
        \operatorname{Sw}_{A_h}
        \bigl(
                \vartheta(z);
                H_0(z),H_s(z)
        \bigr).
\end{equation}
The maximum in \eqref{eq:scalar-readout-gate-scale} is finite because \(\Gamma\) is compact and
the two candidates are continuous.  In particular, \(A_h>0\), and it dominates the gated
difference on every loop state.

\begin{proposition}[Scalar atlas readout]
\label{prop:two-grid-scalar-readout}
Let \(h:[0,1]\to\R\) be CPwL and satisfy \(h(0)=h(1)\).  Then the global CPwL map \(H\)
defined by \eqref{eq:scalar-two-grid-readout} satisfies
\begin{equation}
\label{eq:scalar-atlas-readout-exactness}
        H(\emb(y))
        =
        h(y),
        \qquad
        y\in\T.
\end{equation}
Thus we have
\begin{equation}
\label{eq:iterated-scalar-atlas-readout}
        H(z_n(t))
        =
        h([2^nt]),
        \qquad
        t\in[0,1],\quad n\ge0.
\end{equation}
Moreover, for \(n\ge1\), the map \(t\mapsto H(z_n(t))\) is realized on \([0,1]\) by a ReLU network of
fixed width and depth \(O(n)\).
\end{proposition}

\begin{proof}
Fix \(y\in\T\).  If \(\vartheta(\emb(y))=1\), then
\eqref{eq:scalar-selector-properties} gives \(y\in K_0\), and the switch returns
\[
        H_0(\emb(y))=h(y).
\]
If \(\vartheta(\emb(y))=0\), then \(y\in K_s\), and the switch returns
\(H_s(\emb(y))=h(y)\).

Finally, if \(0<\vartheta(\emb(y))<1\), then \(y\in K_0\cap K_s\).  By
\eqref{eq:scalar-candidate-overlap-agreement}, the two candidates already agree:
\[
        H_0(\emb(y))
        =
        H_s(\emb(y))
        =
        h(y).
\]
The exact-switch lemma therefore returns their common value.  This proves
\eqref{eq:scalar-atlas-readout-exactness}.

Equation \eqref{eq:iterated-scalar-atlas-readout} follows from the exact loop identity
\(z_n(t)=\emb([2^nt])\).
The loop recursion has fixed width and depth \(O(n)\), while \(H\) is a fixed CPwL terminal map,
so their composition has the same asymptotic architecture.
\end{proof}

\section{Finite chart windows and local branch maps}
\label{sec:compatible-fibers}

By the adjoint pairing identity \eqref{eq:invariant-terminal-pairing} from \Cref{subsec:residual-fibers-invariant-cascade}, the refinement
cascade can be evaluated by propagating an intrinsic residual covector forward along the residual
orbit.  The network cannot store its infinite chart representatives, so it records instead finite
coordinate windows in the ordinary and half-shifted charts.  These two windows describe the same
intrinsic covector on their overlap and are related by the corresponding lattice-index shift.

This section develops the finite algebra needed to propagate those windows.  We choose a common
padded index interval \(I\), define the canonical finite chart windows, and prove that the local
adjoint updates are exact and agree on trusted overlaps.  These identities are the algebraic basis
of the atlas adjoint selector constructed in
\Cref{sec:adjoint-branch-selectors}.

\subsection{A padded finite index set}
\label{ss:padded-index-set}

For \(\alpha,\beta\in\mathcal A\), recall from
\Cref{subsec:residual-circle-charts} that a local source--target branch of the doubling map has the
affine form
\[
        R_{\alpha,\beta,q}(x)=2x+2\beta-\alpha-q.
\]
Away from its cuts, the branch label is given by
\(q=Q_{\alpha,\beta}(x)=\lfloor 2x+2\beta-\alpha\rfloor\).

The genuine branch labels associated with the canonical chart coordinates
\(\alpha,\beta\in\mathcal A\) and \(x\in[0,1)\) form the finite set
\[
        \mathcal Q
        :=
        \bigl\{
                Q_{\alpha,\beta}(x):
                \alpha,\beta\in\mathcal A,\ x\in[0,1)
        \bigr\}
        =
        \{-1,0,1,2\}.
\]
In the lift-change calculation underlying
\eqref{eq:infinite-overlap-agreement}, the same target point is described using two lifts that
differ by an integer \(\kappa\).  The corresponding branch labels then satisfy
\(q'=q+2\kappa\).  To accommodate these shifted labels, as well as their finite-window
counterparts below, we use the enlarged finite set
\[
        \mathcal Q^\sharp
        :=
        \mathcal Q+\{-2,0,2\}
        =
        \{-3,-2,\ldots,4\},
\]
since the two-chart lift changes have \(\kappa\in\{-1,0,1\}\).

We now choose the finite coordinate window stored by the network.  We index the unit cells meeting
the interior of the preserved support window by
\[
        I_{\rm phys}:=\{0,1,\ldots,L-1\}.
\]
The adjoint branch formulas may read additional target coordinates, and a change between the
ordinary and half-shifted charts may reindex those coordinates by one.  We therefore enlarge
\(I_{\rm phys}\) to an interval \(I\subset\Z\) such that every target index contributing to an
\(I\)-coordinate of a relevant adjoint update in
\eqref{eq:infinite-adjoint-branch-identity-coord} lies at least one lattice site inside \(I\).

Let \(\mathcal J=\{j\in\Z:c_j\neq0\}\) and set \(j_-=\min\mathcal J\) and
\(j_+=\max\mathcal J\).  Define
\begin{equation}
\label{eq:padded-index-choice}
        i_-:=\min\{-1,j_- -6\},
        \qquad
        i_+:=\max\{L+1,j_+ +5\},
        \qquad
        I:=\{i_-,i_-+1,\ldots,i_+\}.
\end{equation}
This in particular ensures
\[
        I_{\rm phys}\subset I,
        \qquad
        \{-1,L,L+1\}\subset I.
\]
Thus \(I\) contains all physical unit-cell coordinates together with the guard coordinates required
by the adjoint branch updates and the chart shifts.

Let us verify the required interior property.  Suppose that \(q\in\mathcal Q^\sharp\),
\(j\in\mathcal J\), and \(\ell\in I\), and that \(j=q+2i-\ell\) for some \(i\in\Z\).
Equivalently, the target coordinate \(i\) contributes to the source coordinate \(\ell\) in the
adjoint branch formula.  Since \(-3\le q\le4\), the choice
\eqref{eq:padded-index-choice} gives
\[
\begin{aligned}
        i
        &=\frac{j-q+\ell}{2}
        \ge \frac{j_- -4+i_-}{2}
        \ge i_-+1,\\
        i
        &=\frac{j-q+\ell}{2}
        \le \frac{j_+ +3+i_+}{2}
        \le i_+-1.
\end{aligned}
\]
Thus \(i\) belongs to the interior interval
\(I^\circ:=\{i_-+1,\ldots,i_+-1\}\), and consequently
\(i-1,i,i+1\in I\).

Recall that the primal branch matrix maps the source fiber over \(d(y)\) to the target fiber over
\(y\).  Its adjoint therefore maps a target covector over \(y\) to a source covector over \(d(y)\).
The preceding estimate gives the finite-window property needed below.  For every
\(q\in\mathcal Q^\sharp\), every infinite target covector
\(\psi^\infty\in E^*\), and every \(\ell\in I\), one has
\begin{equation}
\label{eq:padded-adjoint-window}
        \bigl(
                (T_{\alpha,\beta,q}^{\infty})^*
                \psi^\infty
        \bigr)_\ell
        =
        \sum_{i\in I}
        c_{q+2i-\ell}\psi_i^\infty .
\end{equation}
Indeed, every index contributing to the corresponding infinite sum lies in \(I^\circ\).
Hence the \(I\)-window of the source covector produced by the adjoint update is determined
entirely by the input \(I\)-window, while the neighboring coordinates needed for chart
reindexing by \(\pm1\) remain available.

We write \(E_I:=\R^I\).  Let
\(\jmath_I:E_I\to E\) denote the canonical inclusion by zero extension, and let
\[
        p_I:E^*\to E_I^*,
        \qquad
        p_I\psi:=\psi\circ\jmath_I,
\]
be the corresponding restriction map.  Under the standard identification
\(E_I^*\cong\R^I\), the covector \(p_I\psi\) is the coordinate tuple
\(p_I\psi=(\psi_i)_{i\in I}\).
We call this tuple the \(I\)-window of the infinite coordinate representative \(\psi\).

For \(\kappa\in\{-1,0,1\}\), define the truncated finite shift
\(C_\kappa^I:E_I\to E_I\) by
\begin{equation}
\label{eq:CI-kappa}
        (C_\kappa^I w)_i
        =
        \begin{cases}
        w_{i+\kappa}, & i+\kappa\in I,\\
        0,            & i+\kappa\notin I.
        \end{cases}
\end{equation}
Its transpose \((C_\kappa^I)^\trans\) acts on \(E_I^*\), and \(C_0^I=\operatorname{Id}\).  
These truncated shifts will be used only on coordinates for which both the original
and shifted indices lie in \(I\).  The canonical chart windows themselves are defined in the next
subsection by forming the appropriate infinite chart representative first and then applying the
restriction \(p_I\).

\subsection{Chart lifts and canonical covector windows}
\label{ss:canonical-chart-representatives}

Let \(y\in\T\), and denote its ordinary representative by
\(\bar y\in[0,1)\).  We extend each chart coordinate to a global half-open convention by choosing
the unique \(x_\beta(y)\in[0,1)\) such that
\[
        y=[\beta+x_\beta(y)],
        \qquad
        \beta\in\mathcal A.
\]
There is then a unique \(\kappa_\beta(y)\in\{0,1\}\) satisfying
\begin{equation}
\label{eq:chart-lift}
        \beta+x_\beta(y)
        =
        \bar y+\kappa_\beta(y).
\end{equation}
For the ordinary chart, \(x_0(y)=\bar y\) and \(\kappa_0(y)=0\).  
For the half-shifted chart, we have
\[
        x_s(y)
        =
        \begin{cases}
        \bar y-s,   & s\le\bar y<1,\\
        \bar y-s+1, & 0\le\bar y<s,
        \end{cases}
        \qquad
        \kappa_s(y)
        =
        \begin{cases}
        0, & s\le\bar y<1,\\
        1, & 0\le\bar y<s.
        \end{cases}
\]
At a chart seam, this convention selects one of its two one-sided lifts.  Exactness assertions will
be imposed only on the corresponding safe arcs, except when seam values are treated explicitly.

Let \(\Phi_y\in\mathcal E_y^*\).  As in
\Cref{subsec:ideal-infinite-chart-mechanism}, denote by
\(\psi_\beta^\infty(y,\Phi_y)\in E^*\) its infinite representative in the chosen lift
\(\beta+x_\beta(y)\).  Relative to the ordinary lift, the dual lift-change rule gives
\begin{equation}
\label{eq:canonical-covector-rep-shift}
        \psi_\beta^\infty(y,\Phi_y)
        =
        C_{-\kappa_\beta(y)}^*
        \psi_0^\infty(y,\Phi_y),
\end{equation}
or in coordinates:
\[
        \bigl(\psi_\beta^\infty(y,\Phi_y)\bigr)_i
        =
        \bigl(\psi_0^\infty(y,\Phi_y)\bigr)_{i+\kappa_\beta(y)},
        \qquad
        i\in\Z.
\]
The canonical finite \(\beta\)-chart window of \(\Phi_y\) is
\begin{equation}
\label{eq:canonical-covector-rep}
        \Rep_\beta^I(y,\Phi_y)
        :=
        p_I\psi_\beta^\infty(y,\Phi_y)
        \in E_I^*.
\end{equation}
Thus the intrinsic lift change is performed in the infinite coordinate space before restriction to
\(I\).  Finite shifts will be used only for coordinates that remain inside the padded window.

\begin{lemma}[Finite chart transition]
\label{lem:finite-chart-transition}
Let \(\beta,\beta'\in\mathcal A\), and suppose that their chosen lifts of \(y\) satisfy
\[
        x_{\beta'}(y)+\beta'
        =
        x_\beta(y)+\beta+\kappa,
        \qquad
        \kappa\in\{-1,0,1\}.
\]
Then we have
\begin{equation}
\label{eq:finite-chart-transition}
        \bigl(\Rep_\beta^I(y,\Phi_y)\bigr)_i
        =
        \bigl(\Rep_{\beta'}^I(y,\Phi_y)\bigr)_{i-\kappa}
\end{equation}
for every \(\Phi_y\in\mathcal E_y^*\), 
whenever \(i,i-\kappa\in I\).
\end{lemma}

\begin{proof}
The two lift parameters satisfy
\(\kappa = \kappa_{\beta'}(y)-\kappa_\beta(y)\).
Writing \(\psi_0^\infty=\psi_0^\infty(y,\Phi_y)\), the coordinate form of
\eqref{eq:canonical-covector-rep-shift} gives
\[
        \bigl(\Rep_{\beta'}^I(y,\Phi_y)\bigr)_{i-\kappa}
        =
        (\psi_0^\infty)_{i-\kappa+\kappa_{\beta'}(y)}
        =
        (\psi_0^\infty)_{i+\kappa_\beta(y)}
        =
        \bigl(\Rep_\beta^I(y,\Phi_y)\bigr)_i ,
\]
completing the proof.
\end{proof}

\begin{remark}
\label{rem:canonical-not-arbitrary}
A safe-exact atlas state consists of finite chart windows of one intrinsic covector.  Where both
charts are safe, \Cref{lem:finite-chart-transition} supplies their compatibility; outside a chart's
safe region, no correctness is required of its stored component.
\end{remark}

\subsection{Finite branch matrices and local adjoint exactness}
\label{ss:finite-branch}

We now restrict the infinite adjoint branch updates to the padded window \(I\).  For
\(\alpha,\beta\in\mathcal A\) and \(q\in\Z\), define
\(T_{\alpha,\beta,q}^I:E_I\to E_I\) by
\begin{equation}
\label{eq:finite-branch-matrix}
        \bigl(T_{\alpha,\beta,q}^I\bigr)_{i\ell}
        :=
        c_{q+2i-\ell},
        \qquad
        i,\ell\in I.
\end{equation}
The entries depend only on \(q\); the chart labels record the local source--target branch in which
the matrix is used.

For \(q\in\mathcal Q^\sharp\), the padding choice from
\Cref{ss:padded-index-set} gives
\begin{equation}
\label{eq:finite-adjoint-window-identity}
        p_I
        \bigl(T_{\alpha,\beta,q}^{\infty}\bigr)^*
        \psi^\infty
        =
        \bigl(T_{\alpha,\beta,q}^I\bigr)^\trans
        p_I\psi^\infty,
        \qquad
        \psi^\infty\in E^*.
\end{equation}
Thus every coordinate in the output \(I\)-window of the infinite adjoint update is determined
exactly by the input \(I\)-window.

\begin{lemma}[Local adjoint exactness for canonical windows]
\label{lem:canonical-local-exactness}
Let \(y\in K_\beta\), set \(y^+:=d(y)\), and assume that \(y^+\in K_\alpha\).  For
\(\Phi_y\in\mathcal E_y^*\), define
\[
        \Phi_{y^+}
        :=
        \mathcal B_y^*\Phi_y.
\]
Let \(x:=x_\beta(y)\), \(r:=x_\alpha(y^+)\), and let \(q\in\Z\) be determined by
\(2(x+\beta)=q+\alpha+r\).
Then we have
\begin{equation}
\label{eq:canonical-local-exactness}
        \bigl(T_{\alpha,\beta,q}^I\bigr)^\trans
        \Rep_\beta^I(y,\Phi_y)
        =
        \Rep_\alpha^I(y^+,\Phi_{y^+}).
\end{equation}
\end{lemma}

\begin{proof}
With
\(\psi_\beta^\infty := \psi_\beta^\infty(y,\Phi_y)\) and
\(\psi_\alpha^{+,\infty} := \psi_\alpha^\infty(y^+,\Phi_{y^+})\),
the identity \eqref{eq:infinite-adjoint-branch-identity} becomes
\[
        \psi_\alpha^{+,\infty}
        =
        \bigl(T_{\alpha,\beta,q}^\infty\bigr)^*
        \psi_\beta^\infty.
\]
Since \(x,r\in[0,1)\), the branch label \(q\) belongs to
\(\mathcal Q\subset\mathcal Q^\sharp\).  Applying
\eqref{eq:finite-adjoint-window-identity} gives
\[
        \bigl(T_{\alpha,\beta,q}^I\bigr)^\trans
        \Rep_\beta^I(y,\Phi_y)
        =
        \bigl(T_{\alpha,\beta,q}^I\bigr)^\trans
        p_I\psi_\beta^\infty
        =
        p_I
        \bigl(T_{\alpha,\beta,q}^\infty\bigr)^*
        \psi_\beta^\infty
        =
        p_I\psi_\alpha^{+,\infty}
        =
        \Rep_\alpha^I(y^+,\Phi_{y^+}) ,
\]
establishing the proof.
\end{proof}

Hence the finite branch transpose propagates the canonical chart windows exactly whenever the
target and source charts are safe.  Equivalently, one may regard the adjoint evolution as taking
place entirely in the infinite-dimensional intrinsic fibers: the finite branch matrices merely
compute the required \(I\)-windows of that evolution.

\subsection{Finite overlap agreement}
\label{ss:finite-overlap-agreement}

We now establish the finite identity that supplies the agreement regions for the atlas selector.
When two compatible target-chart windows describe the same intrinsic covector, their local adjoint
updates agree after both candidates have been written in a common source chart.

\begin{lemma}[Finite padded overlap agreement]
\label{lem:finite-padded-overlap-agreement}
Fix a source chart \(\alpha\in\mathcal A\), and let
\(\beta,\beta'\in\mathcal A\) be two target charts.  Suppose that their chosen lifts of the same
target circle point satisfy
\begin{equation}
\label{eq:target-chart-lift-difference}
        x_{\beta'}+\beta'
        =
        x_\beta+\beta+\kappa,
        \qquad
        \kappa\in\{-1,0,1\}.
\end{equation}
Assume that both target descriptions use the same source coordinate \(x_\alpha\), so that
\[
        2(x_\beta+\beta)=q+\alpha+x_\alpha,
        \qquad
        2(x_{\beta'}+\beta')=q'+\alpha+x_\alpha,
\]
for some \(q,q'\in\mathcal Q^\sharp\).  
Then we have
\begin{equation}
\label{eq:q-prime-shift}
        q'=q+2\kappa.
\end{equation}

Let \(\psi_\beta,\psi_{\beta'}\in E_I^*\) satisfy the finite chart compatibility relation
\begin{equation}
\label{eq:finite-covector-compatibility}
        (\psi_\beta)_i
        =
        (\psi_{\beta'})_{i-\kappa}
\end{equation}
whenever \(i,i-\kappa\in I\).  Then their adjoint branch updates into the common source chart
\(\alpha\) agree:
\begin{equation}
\label{eq:finite-overlap-agreement}
        (T_{\alpha,\beta,q}^I)^\trans\psi_\beta
        =
        (T_{\alpha,\beta',q'}^I)^\trans\psi_{\beta'}.
\end{equation}
\end{lemma}

\begin{proof}
Doubling \eqref{eq:target-chart-lift-difference} and using the first source--target relation gives
\[
        2(x_{\beta'}+\beta')
        =
        (q+2\kappa)+\alpha+x_\alpha.
\]
Comparison with the second relation proves \eqref{eq:q-prime-shift}.

Fix \(\ell\in I\), and introduce the sets of indices that contribute to the two adjoint sums:
\[
        A_\ell
        =
        \{i\in I:q+2i-\ell\in\mathcal J\},
        \qquad\text{and}\qquad
        A'_\ell
        =
        \{r\in I:q'+2r-\ell\in\mathcal J\}.
\]
If \(i\in A_\ell\), the padding property from \Cref{ss:padded-index-set} gives
\(i-1,i,i+1\in I\), and therefore \(i-\kappa\in I\).  Moreover, we have
\[
        q'+2(i-\kappa)-\ell
        =
        q+2i-\ell
        \in\mathcal J,
\]
so \(i-\kappa\in A'_\ell\).

Conversely, if \(r\in A'_\ell\), the same padding property gives \(r+\kappa\in I\), and
\[
        q+2(r+\kappa)-\ell
        =
        q'+2r-\ell
        \in\mathcal J.
\]
Hence \(r+\kappa\in A_\ell\).  Thus \(i\mapsto i-\kappa\) is a bijection from
\(A_\ell\) onto \(A'_\ell\).

Using the compatibility relation \eqref{eq:finite-covector-compatibility}, we now obtain
\[
\begin{aligned}
\bigl((T_{\alpha,\beta,q}^I)^\trans\psi_\beta\bigr)_\ell
&=
\sum_{i\in A_\ell}
        c_{q+2i-\ell}(\psi_\beta)_i
=
\sum_{i\in A_\ell}
        c_{q+2i-\ell}(\psi_{\beta'})_{i-\kappa}\\
&=
\sum_{r\in A'_\ell}
        c_{q'+2r-\ell}(\psi_{\beta'})_r
=
\bigl((T_{\alpha,\beta',q'}^I)^\trans\psi_{\beta'}\bigr)_\ell.
\end{aligned}
\]
Since \(\ell\in I\) was arbitrary, \eqref{eq:finite-overlap-agreement} follows.
\end{proof}

The compatibility condition leaves at most one endpoint coordinate of each finite window
unconstrained when \(\kappa=\pm1\).  This causes no ambiguity in the preceding lemma: by the
padding property, every coordinate used by either adjoint branch update has its shifted counterpart
inside \(I\).  The unmatched endpoint values are therefore inactive in both updates.

\begin{corollary}[Canonical-window overlap agreement]
\label{cor:canonical-overlap-agreement}
Let \(y\in K_\beta\cap K_{\beta'}\), and suppose that
\(y^+=d(y)\in K_\alpha\).  Let \(\Phi_y\in\mathcal E_y^*\), and define
\(\Phi_{y^+}=\mathcal B_y^*\Phi_y\).  Set
\[
        \psi_\beta=\Rep_\beta^I(y,\Phi_y),
        \qquad
        \psi_{\beta'}=\Rep_{\beta'}^I(y,\Phi_y).
\]
Let \(q\) and \(q'\) be the branch labels determined by
\[
        2(x_\beta(y)+\beta)
        =
        q+\alpha+x_\alpha(y^+)
\qquad\text{and}\qquad
        2(x_{\beta'}(y)+\beta')
        =
        q'+\alpha+x_\alpha(y^+).
\]
Then we have
\begin{equation}
\label{eq:canonical-overlap-agreement}
        (T_{\alpha,\beta,q}^I)^\trans\psi_\beta
        =
        (T_{\alpha,\beta',q'}^I)^\trans\psi_{\beta'}
        =
        \Rep_\alpha^I(y^+,\Phi_{y^+}).
\end{equation}
\end{corollary}

\begin{proof}
The canonical target windows satisfy
\eqref{eq:finite-covector-compatibility} by
\Cref{lem:finite-chart-transition}.  Hence
\Cref{lem:finite-padded-overlap-agreement} gives equality of the two branch candidates.
Applying \Cref{lem:canonical-local-exactness} to either target chart identifies their common value
with \(\Rep_\alpha^I(y^+,\Phi_{y^+})\).
\end{proof}

This agreement is the exact switching condition used by the atlas adjoint selector below.

\section{The atlas adjoint selector}
\label{sec:adjoint-branch-selectors}

We now construct the one-step atlas adjoint selector.  Its input is a redundant finite state
\(\Psi=(\psi_0,\psi_s)\in E_I^*\oplus E_I^*\), with one coordinate window for each chart.
Each component is required to represent the underlying intrinsic covector only where its chart is
safe.

More precisely, for \(y\in\T\), we say that \(\Psi\) is \emph{safe-exact over \(y\)} if there exists
an intrinsic covector \(\Phi_y\in\mathcal E_y^*\) such that, for every
\(\beta\in\mathcal A\),
\begin{equation}
\label{eq:safe-exact-state-condition}
        y\in K_\beta
        \quad\Longrightarrow\quad
        \psi_\beta=\Rep_\beta^I(y,\Phi_y).
\end{equation}
When \(y\in K_0\cap K_s\), both components are therefore canonical finite windows of the same
intrinsic covector and satisfy the chart-transition identity of
\Cref{lem:finite-chart-transition}.  No condition is imposed on \(\psi_\beta\) when
\(y\notin K_\beta\).

The selector propagates such a state from \(y\) to \(y^+=d(y)\).  For each source chart
\(\alpha\in\mathcal A\), it forms one local adjoint candidate from each target chart at \(y\).
Where the target-chart selector takes a plateau value \(0\) or \(1\), it selects the corresponding
candidate directly.  Its fractional region lies inside \(K_0\cap K_s\), where both target windows
are canonical representatives of \(\Phi_y\).  By
\Cref{cor:canonical-overlap-agreement}, the two local updates into chart \(\alpha\) then agree and,
whenever \(y^+\in K_\alpha\), their common value is
\[
        \Rep_\alpha^I\bigl(y^+,\mathcal B_y^*\Phi_y\bigr).
\]
Thus the ReLU switch preserves the intrinsic adjoint update on every trusted overlap.

\subsection{Local CPwL adjoint update modules}
\label{ss:local-cpwl-adjoint-modules}

We now construct the local network modules that implement the finite adjoint branch update
\(\psi\mapsto(T_{\alpha,\beta,q}^I)^\trans\psi\)
from the loop state, without evaluating the discontinuous branch label \(q\) globally.  For each
\(\alpha,\beta\in\mathcal A\) and each gate scale \(A>0\), the module
\(\mathcal U_{\alpha,\beta,A}\) combines the finitely many fixed branch matrices using CPwL
plateau gates.  On every trusted source--target branch region, it returns the corresponding exact
adjoint update.  The scale \(A\) will later be chosen uniformly for all gates occurring in the
\(n\)-step recursion.

For fixed \(\alpha\) and \(\beta\), define the {\em trusted region}
\[
        \Gamma_{\alpha,\beta}
        :=
        \{\emb(y):y\in K_\beta,\ d(y)\in K_\alpha\}
        \subset\Gamma.
\]
This is a finite union of closed polygonal arcs, possibly including degenerate components.  On
each connected component, the target coordinate \(x_\beta(y)\) and the source coordinate
\(x_\alpha(d(y))\) are continuous.  The branch label \(q\) determined by
\begin{equation}
\label{eq:trusted-branch-label}
        2\bigl(x_\beta(y)+\beta\bigr)
        =
        q+\alpha+x_\alpha(d(y))
\end{equation}
is therefore a continuous integer-valued function and hence is constant on that component.

\begin{lemma}[Local CPwL update modules]
\label{lem:local-cpwl-update-modules}
Fix \(\alpha,\beta\in\mathcal A\) and \(A>0\).  There exists a globally defined CPwL map
\[
        \mathcal U_{\alpha,\beta,A}:
        \R^2\times E_I^*\to E_I^*
\]
with the following property.  Let \(z=\emb(y)\in\Gamma_{\alpha,\beta}\), and let \(q\) be the
branch label on the trusted component containing \(z\).  Then
\begin{equation}
\label{eq:local-module-exact-branch}
        \mathcal U_{\alpha,\beta,A}(z,\psi)
        =
        (T_{\alpha,\beta,q}^I)^\trans\psi
\end{equation}
whenever \(\psi\in E_I^*\) satisfies
\[
        \|L_\nu\psi\|_\infty\le A
\]
for every finite branch map \(L_\nu\) used in the construction.

In particular, let \(\Phi_y\in\mathcal E_y^*\), set
\(\Phi_{d(y)}=\mathcal B_y^*\Phi_y\), and suppose that
\(\psi=\Rep_\beta^I(y,\Phi_y)\).  Under the same gate bounds, one has
\begin{equation}
\label{eq:local-module-canonical-exactness}
        \mathcal U_{\alpha,\beta,A}
        \bigl(\emb(y),\Rep_\beta^I(y,\Phi_y)\bigr)
        =
        \Rep_\alpha^I\bigl(d(y),\Phi_{d(y)}\bigr).
\end{equation}
No exactness is asserted outside \(\Gamma_{\alpha,\beta}\).
\end{lemma}

\begin{proof}
Write the connected components of \(\Gamma_{\alpha,\beta}\) as
\(\Gamma_1,\ldots,\Gamma_N\), and let \(q_\nu\) denote the branch label on \(\Gamma_\nu\).
Set
\[
        L_\nu=(T_{\alpha,\beta,q_\nu}^I)^\trans,
        \qquad \nu=1,\ldots,N.
\]
Distinct trusted components are disjoint compact subsets of the polygonal loop and are separated
by complementary arcs.  After subdividing \(\Gamma\) at their endpoints, we may therefore choose
CPwL plateau functions \(\eta_\nu:\Gamma\to[0,1]\) satisfying
\[
        \eta_\nu=1\ \text{on }\Gamma_\nu,
        \qquad
        \eta_\nu=0\ \text{on }\Gamma_\mu
        \quad(\mu\ne\nu).
\]
On each complementary arc, choose any CPwL interpolation between the prescribed endpoint
values.  Extend these functions to global CPwL maps on \(\R^2\), clipping the extensions to
\([0,1]\) if necessary.

Define the local module componentwise by
\begin{equation}
\label{eq:local-module-plateau-formula}
        \mathcal U_{\alpha,\beta,A}(z,\psi)
        :=
        \sum_{\nu=1}^N
        \Pi_A\bigl(\eta_\nu(z),L_\nu\psi\bigr).
\end{equation}
This is a globally defined CPwL map on \(\R^2\times E_I^*\).

Suppose that \(z\in\Gamma_j\).  Then \(\eta_j(z)=1\), while
\(\eta_\nu(z)=0\) for \(\nu\ne j\).  Under the stated bounds, the gate identities give
\[
        \Pi_A(1,L_j\psi)=L_j\psi,
        \qquad
        \Pi_A(0,L_\nu\psi)=0
        \quad(\nu\ne j).
\]
Thus we have
\[
        \mathcal U_{\alpha,\beta,A}(z,\psi)
        =
        L_j\psi
        =
        (T_{\alpha,\beta,q_j}^I)^\trans\psi,
\]
which proves \eqref{eq:local-module-exact-branch}.

For the canonical statement, take
\(\psi=\Rep_\beta^I(y,\Phi_y)\).  Since \(y\in K_\beta\),
\(d(y)\in K_\alpha\), and \(q_j\) is the corresponding local branch label,
\Cref{lem:canonical-local-exactness} yields
\[
        (T_{\alpha,\beta,q_j}^I)^\trans
        \Rep_\beta^I(y,\Phi_y)
        =
        \Rep_\alpha^I
        \bigl(d(y),\mathcal B_y^*\Phi_y\bigr).
\]
Combining this identity with
\eqref{eq:local-module-exact-branch} proves
\eqref{eq:local-module-canonical-exactness}.
\end{proof}

\subsection{One-step atlas adjoint update}
\label{ss:atlas-adjoint-selector}

Fix a gate scale \(A>0\), and let
\(\vartheta:\R^2\to[0,1]\) be the chart selector from
\Cref{subsec:scalar-switch}.  On the loop, the values
\(\vartheta=1\) and \(\vartheta=0\) select the ordinary and half-shifted target charts,
respectively, while fractional values occur only in \(K_0\cap K_s\).

For each source chart \(\alpha\), the two target-chart components produce candidates in the same
finite coordinate space \(E_I^*\).  On the fractional region of \(\vartheta\), both target charts
are safe, so their compatible windows yield identical candidates by
\Cref{cor:canonical-overlap-agreement}.  They may therefore be combined by the exact ReLU
switch.

More precisely, given \(z\in\R^2\) and
\(\Psi=(\psi_0,\psi_s)\in E_I^*\oplus E_I^*\), define
\[
        Y_{\alpha,0}
        :=
        \mathcal U_{\alpha,0,A}(z,\psi_0),
        \qquad
        Y_{\alpha,s}
        :=
        \mathcal U_{\alpha,s,A}(z,\psi_s),
        \qquad
        \alpha\in\mathcal A.
\]
The updated atlas state
\(\mathcal S_A(z,\Psi)\in E_I^*\oplus E_I^*\) is given componentwise by
\begin{equation}
\label{eq:atlas-selector-component}
        \bigl(\mathcal S_A(z,\Psi)\bigr)_\alpha
        :=
        \Sw_A\bigl(
                \vartheta(z);
                Y_{\alpha,0},Y_{\alpha,s}
        \bigr),
        \qquad
        \alpha\in\mathcal A.
\end{equation}
Since the local update modules, the chart selector, and the exact switch are CPwL, the resulting
map
\[
        \mathcal S_A:
        \R^2\times(E_I^*\oplus E_I^*)
        \longrightarrow
        E_I^*\oplus E_I^*
\]
is globally CPwL.

\begin{proposition}[One-step propagation of safe-exact states]
\label{prop:safe-exact-compatible-propagation}
Let \(z=\emb(y)\), and suppose that
\(\Psi=(\psi_0,\psi_s)\) is safe-exact over \(y\), witnessed by an intrinsic covector
\(\Phi_y\in\mathcal E_y^*\).  Set
\(y^+=d(y)\) and \(\Phi_{y^+}=\mathcal B_y^*\Phi_y\),
and define
\[
        \Psi^+
        =
        (\psi_0^+,\psi_s^+)
        :=
        \mathcal S_A(z,\Psi).
\]
Assume that the gate bounds in
\Cref{lem:local-cpwl-update-modules} hold for every local module evaluated in this step, and that
the outer switches satisfy
\begin{equation}
\label{eq:one-step-outer-gate-bound}
        \|Y_{\alpha,s}-Y_{\alpha,0}\|_\infty\le A,
        \qquad
        \alpha\in\mathcal A.
\end{equation}
Then \(\Psi^+\) is safe-exact over \(y^+\), witnessed by \(\Phi_{y^+}\).  Equivalently, for every
\(\alpha\in\mathcal A\), one has
\begin{equation}
\label{eq:propagated-safe-exactness}
        y^+\in K_\alpha
        \quad\Longrightarrow\quad
        \psi_\alpha^+
        =
        \Rep_\alpha^I(y^+,\Phi_{y^+}).
\end{equation}
If \(y^+\in K_0\cap K_s\), the two propagated components satisfy the finite chart-transition
identity at \(y^+\).
\end{proposition}

\begin{proof}
Fix \(\alpha\in\mathcal A\) such that \(y^+\in K_\alpha\), and write
\(\theta:=\vartheta(z)\).

Suppose first that \(\theta=1\).  By
\eqref{eq:scalar-selector-properties}, one has \(y\in K_0\).  Safe-exactness gives
\(\psi_0=\Rep_0^I(y,\Phi_y)\), and
\Cref{lem:local-cpwl-update-modules} yields
\[
        Y_{\alpha,0}
        =
        \Rep_\alpha^I(y^+,\Phi_{y^+}).
\]
The outer switch selects \(Y_{\alpha,0}\), and hence
\(\psi_\alpha^+=\Rep_\alpha^I(y^+,\Phi_{y^+})\).

If \(\theta=0\), then \(y\in K_s\).  The same argument in the half-shifted target chart gives
\[
        Y_{\alpha,s}
        =
        \Rep_\alpha^I(y^+,\Phi_{y^+}),
\]
and the outer switch returns \(Y_{\alpha,s}\).

It remains to consider \(0<\theta<1\).  In this case,
\eqref{eq:scalar-selector-properties} gives \(y\in K_0\cap K_s\).  Thus safe-exactness gives
\[
        \psi_0=\Rep_0^I(y,\Phi_y),
        \qquad
        \psi_s=\Rep_s^I(y,\Phi_y).
\]
Since \(y^+\in K_\alpha\), both local modules are evaluated on trusted regions.  Their canonical
exactness gives
\[
        Y_{\alpha,0}
        =
        Y_{\alpha,s}
        =
        \Rep_\alpha^I(y^+,\Phi_{y^+}).
\]
The outer switch therefore returns this common value.  This proves
\eqref{eq:propagated-safe-exactness} in all cases.

If \(y^+\in K_0\cap K_s\), the preceding result gives
\[
        \psi_0^+=\Rep_0^I(y^+,\Phi_{y^+}),
        \qquad
        \psi_s^+=\Rep_s^I(y^+,\Phi_{y^+}).
\]
Their chart-transition relation follows from
\Cref{lem:finite-chart-transition}.  Hence \(\Psi^+\) is safe-exact over \(y^+\).
\end{proof}

\subsection{Recursive propagation and gate scales}
\label{ss:recursive-atlas-adjoint-propagation}

Fix \(n\ge1\) and \(t\in[0,1]\).  
Write \(y_m=d^m([t])\) and \(z_m(t)=\emb(y_m)\) for \(m=0,\ldots,n\).
Thus \(z_0(t)=\emb([t])\) and \(z_{m+1}(t)=D(z_m(t))\).

Suppose that
\(\Psi_0(t)=\bigl(\psi_{0,0}(t),\psi_{s,0}(t)\bigr)\in E_I^*\oplus E_I^*\)
is safe-exact over \(y_0=[t]\), witnessed by
\(\Phi_0(t)\in\mathcal E_{y_0}^*\).  Define the intrinsic adjoint orbit by
\begin{equation}
\label{eq:recursive-intrinsic-adjoint}
        \Phi_{m+1}(t)
        :=
        \mathcal B_{y_m}^*\Phi_m(t),
        \qquad
        m=0,\ldots,n-1.
\end{equation}

For \(\Psi=(\psi_0,\psi_s)\), write
\[
        \|\Psi\|_\infty
        :=
        \max\{\|\psi_0\|_\infty,\|\psi_s\|_\infty\}.
\]
The fixed local modules and the gate contraction estimate
\eqref{eq:gate-contraction} give constants
\(C_{\rm gate},M_{\rm sel}\ge1\), depending only on the atlas construction and not on the gate
scale \(A\), such that
\begin{align}
\label{eq:atlas-gate-input-bound}
        \text{every gate argument in }\mathcal S_A(z,\Psi)
        &\text{ has magnitude at most }
        C_{\rm gate}\|\Psi\|_\infty,\\
\label{eq:selector-bound}
        \|\mathcal S_A(z,\Psi)\|_\infty
        &\le
        M_{\rm sel}\|\Psi\|_\infty.
\end{align}
Indeed, each local candidate is a finite sum of gated fixed linear branch maps, and the outer
selector is another application of the same contraction estimate.

Assume that
\[
        M_0
        =
        \sup_{\tau\in[0,1]}
        \|\Psi_0(\tau)\|_\infty
        <\infty,
\]
and choose
\begin{equation}
\label{eq:atlas-gate-scale}
        A_n
        =
        C_{\rm gate}M_{\rm sel}^n\max\{1,M_0\}.
\end{equation}
The finite atlas recursion is then
\begin{equation}
\label{eq:atlas-adjoint-recursion}
        \Psi_{m+1}(t)
        =
        \mathcal S_{A_n}\bigl(z_m(t),\Psi_m(t)\bigr),
        \qquad
        m=0,\ldots,n-1.
\end{equation}

\begin{lemma}[Uniform gate scale]
\label{lem:atlas-uniform-gate-scale}
Every internal local-module gate and every outer atlas switch in
\eqref{eq:atlas-adjoint-recursion} is evaluated within the scale \(A_n\).
\end{lemma}

\begin{proof}
The unconditional selector bound gives
\[
        \|\Psi_m(t)\|_\infty
        \le
        M_{\rm sel}^m\|\Psi_0(t)\|_\infty
        \le
        M_{\rm sel}^n\max\{1,M_0\},
        \qquad
        0\le m\le n.
\]
Hence \eqref{eq:atlas-gate-input-bound} bounds every gate argument at a stage \(m<n\) by
\(A_n\).
\end{proof}

\begin{proposition}[Correctness of the recursive atlas adjoint]
\label{prop:recursive-atlas-adjoint-correctness}
For every \(m=0,\ldots,n\), the state
\[
        \Psi_m(t)
        =
        \bigl(\psi_{0,m}(t),\psi_{s,m}(t)\bigr)
\]
is safe-exact over \(y_m\), witnessed by the intrinsic covector \(\Phi_m(t)\).  
In particular,
\begin{equation}
\label{eq:recursive-safe-chart-correctness}
        y_m\in K_\beta
        \quad\Longrightarrow\quad
        \psi_{\beta,m}(t)
        =
        \Rep_\beta^I\bigl(y_m,\Phi_m(t)\bigr),
        \qquad
        \beta\in\mathcal A.
\end{equation}
On \(K_0\cap K_s\), the two components satisfy the finite chart-transition identity.
\end{proposition}

\begin{proof}
The assertion holds at \(m=0\) by hypothesis.  If it holds at stage \(m\), then
\Cref{lem:atlas-uniform-gate-scale} places all gates in their exact ranges, and
\Cref{prop:safe-exact-compatible-propagation} shows that
\[
        \Psi_{m+1}(t)
        =
        \mathcal S_{A_n}\bigl(z_m(t),\Psi_m(t)\bigr)
\]
is safe-exact over \(y_{m+1}=d(y_m)\), witnessed by
\(\Phi_{m+1}(t)=\mathcal B_{y_m}^*\Phi_m(t)\).  Induction proves
\eqref{eq:recursive-safe-chart-correctness}.  When both charts are safe, their components are
canonical windows of the same intrinsic covector, so
\Cref{lem:finite-chart-transition} gives their compatibility.
\end{proof}

\subsection{Network realization of the atlas recursion}
\label{ss:network-realization-atlas-recursion}

The initial atlas state may require a realization whose depth depends on \(n\).  We denote this
initial cost by \(D_0(n)\); the atlas recursion adds only \(O(n)\) further depth.

\begin{theorem}[Fixed-width realization of the atlas adjoint recursion]
\label{thm:fixed-width-atlas-adjoint-recursion}
Fix the mask, the support window, the padded index set \(I\), and the two-chart atlas.  For
\(t\in[0,1]\) and \(0\le m\le n\), set
\(y_m=[2^m t]\).
Suppose that the initial map
\(t\mapsto\Psi_0(t)\in E_I^*\oplus E_I^*\)
is realized on \([0,1]\) by ReLU networks of width bounded independently of \(n\) and depth
\(D_0(n)\).  Assume that \(\Psi_0(t)\) is safe-exact over \(y_0\), witnessed by
\(\Phi_0(t)\in\mathcal E_{y_0}^*\), and that
\[
        \sup_{\tau\in[0,1]}\|\Psi_0(\tau)\|_\infty\le M_0
\]
for some \(M_0\) independent of \(n\).

Choose \(A_n\) by \eqref{eq:atlas-gate-scale}, and define the intrinsic and finite recursions by
\eqref{eq:recursive-intrinsic-adjoint} and \eqref{eq:atlas-adjoint-recursion}.  Then the terminal
map \(t\mapsto\Psi_n(t)\) is realized on \([0,1]\) by a ReLU network of width bounded
independently of \(n\) and depth
\(D_0(n)+O(n)\).
Moreover, writing
\(\Psi_n(t)=(\psi_{0,n}(t),\psi_{s,n}(t))\), one has
\begin{equation}
\label{eq:network-final-safe-exactness}
        y_n\in K_\beta
        \quad\Longrightarrow\quad
        \psi_{\beta,n}(t)
        =
        \Rep_\beta^I\bigl(y_n,\Phi_n(t)\bigr),
        \qquad
        \beta\in\mathcal A.
\end{equation}
\end{theorem}

\begin{proof}
The joint initialization
\(t\mapsto\bigl(\emb([t]),\Psi_0(t)\bigr)\)
has width bounded independently of \(n\) and depth \(D_0(n)+O(1)\).
For fixed \(n\), one recursion step is the CPwL map
\[
        (z,\Psi)
        \longmapsto
        \bigl(D(z),\mathcal S_{A_n}(z,\Psi)\bigr).
\]
It is assembled from the fixed loop controller, finitely many fixed linear branch maps and
plateau functions, and ReLU gates whose numerical scale is \(A_n\).  Its width and depth are
therefore bounded independently of \(n\).  Repeating this block \(n\) times realizes
\(\Psi_n\) with depth \(D_0(n)+O(n)\) and fixed width.
All gates are exact by \Cref{lem:atlas-uniform-gate-scale}, and
\eqref{eq:network-final-safe-exactness} follows from
\Cref{prop:recursive-atlas-adjoint-correctness}.
\end{proof}

\section{Reference rank-one terminal atoms}
\label{sec:rank-one-terminal-atoms}

We now combine the scalar atlas readout of \Cref{sec:scalar-readout} with the atlas adjoint
recursion of \Cref{sec:adjoint-branch-selectors}.  This section proves the realization theorem for
rank-one atoms whose scalar profiles are supported inside a fixed region where the ordinary chart
is safe.  In \Cref{sec:general-cpwl-seeds}, arbitrary compactly supported CPwL seeds will be
reduced to finitely many translated atoms of this form.

Fix a closed interval \(J_{\rm ref}\Subset(0,1)\) such that
\begin{equation}
\label{eq:reference-interval-ordinary}
        t\in J_{\rm ref}
        \quad\Longrightarrow\quad
        [t]\in\operatorname{int}K_0.
\end{equation}
Let \(h\in\CPwL([0,1])\) satisfy \(\supp h\subset J_{\rm ref}\), and let
\(u\in E_I\).  Write \(u^\infty=\jmath_Iu\in E\) for its zero extension.  
Define a section \(W\) by prescribing its representative in the ordinary lift
\(\bar y\in[0,1)\):
\begin{equation}
\label{eq:reference-atom-ordinary-representative}
        w^\infty(y)
        :=
        h(\bar y)u^\infty,
        \qquad
        y\in\T.
\end{equation}
Because \(h\) vanishes in neighborhoods of both endpoints, this representative is zero near the
ordinary seam and therefore defines a well-defined CPwL section of the residual local system \(\mathcal E\).
We call \(W\) a \emph{reference rank-one terminal atom}.

Define the refinements of \(W_0=W\) intrinsically by
\begin{equation}
\label{eq:reference-atom-primal-cascade}
        W_{m+1}(y)
        =
        \mathcal B_y W_m(d(y)),
        \qquad
        m\ge0.
\end{equation}
No condition is imposed on \(u\) beyond \(u\in E_I\).  The localization
\(\supp h\subset J_{\rm ref}\) ensures that whenever the terminal scalar factor is nonzero, the
ordinary chart is safe, so the terminal contribution can be recovered by a direct ordinary-chart
pairing.  The atlas selector remains confined to the adjoint recursion.

\subsection{The reference-atom adjoint cascade}
\label{ss:rank-one-cascade}

Fix an output covector \(\lambda\in E_I^*\), and let
\(\lambda^\infty\in E^*\) be its zero extension.  We work with a general \(\lambda\) because the
adjoint argument is linear; in the coordinate realization used later, \(\lambda\) will be chosen
from the standard coordinate covectors of \(E_I^*\).

For \(t\in[0,1]\), set \(y_0=[t]\), and let
\(\Lambda_0(t)\in\mathcal E_{y_0}^*\) be the intrinsic covector whose representative in the lift
\(t\) is \(\lambda^\infty\).  Define
\begin{equation}
\label{eq:rank-one-target-function}
        F_{n,\lambda,W}(t)
        :=
        \bigl\langle
                \Lambda_0(t),W_n(y_0)
        \bigr\rangle .
\end{equation}
Thus \(F_{n,\lambda,W}(t)\) is the linear readout of the ordinary chart representative of
\(W_n(y_0)\) determined by \(\lambda\).  For a standard coordinate covector, it extracts the
corresponding component.

For \(0\le m\le n\), set \(y_m=d^m(y_0)=[2^m t]\), and let \(z_m(t)=\emb(y_m)\)
be the corresponding loop state.  Since \(h(0)=h(1)=0\), the profile \(h\) defines a CPwL
function on \(\T\).  If \(H\) denotes its scalar atlas readout from
\Cref{prop:two-grid-scalar-readout}, define the terminal scalar
\(H_n(t):=H(z_n(t))\).  The exactness of that readout gives
\(H_n(t)=h(y_n)=h(\bar y_n)\).

Starting from \(\Lambda_0(t)\), define the intrinsic adjoint cascade by
\begin{equation}
\label{eq:rank-one-unscaled-adjoint-cascade}
        \Lambda_{m+1}(t)
        =
        \mathcal B_{y_m}^*\Lambda_m(t),
        \qquad
        m=0,\ldots,n-1.
\end{equation}
The pairing identity
\eqref{eq:invariant-terminal-pairing} gives
\[
        F_{n,\lambda,W}(t)
        =
        \bigl\langle
                \Lambda_n(t),W(y_n)
        \bigr\rangle .
\]
Let \(\lambda_n^\infty(t)\in E^*\) denote the representative of
\(\Lambda_n(t)\) in the ordinary lift \(\bar y_n\).  By
\eqref{eq:reference-atom-ordinary-representative}, we obtain
\begin{equation}
\label{eq:rank-one-terminal-factorization}
        F_{n,\lambda,W}(t)
        =
        H_n(t)
        \bigl\langle
                \lambda_n^\infty(t),u^\infty
        \bigr\rangle .
\end{equation}

The scalar factor \(H_n(t)\) depends on the terminal residual point, whereas the adjoint cascade
uses the residual states \(y_0,\ldots,y_{n-1}\).  To avoid multiplying two variable network outputs
after the recursion, we absorb \(H_n(t)\) into the initial covector.  Define
\begin{equation}
\label{eq:rank-one-initial-covector}
        \Phi_0(t)
        =
        H_n(t)\Lambda_0(t)
        \in\mathcal E_{y_0}^*,
\end{equation}
and propagate it intrinsically by
\begin{equation}
\label{eq:rank-one-scaled-adjoint-cascade}
        \Phi_{m+1}(t)
        =
        \mathcal B_{y_m}^*\Phi_m(t),
        \qquad
        m=0,\ldots,n-1.
\end{equation}
By linearity, \(\Phi_m(t)=H_n(t)\Lambda_m(t)\) for every \(m\).  Let
\(\phi_n^\infty(t)\) denote the representative of \(\Phi_n(t)\) in the ordinary lift
\(\bar y_n\).  Then
\[
        \phi_n^\infty(t)
        =
        H_n(t)\lambda_n^\infty(t).
\]
Since \(u^\infty\) is supported in \(I\), the factorization
\eqref{eq:rank-one-terminal-factorization} becomes
\begin{equation}
\label{eq:rank-one-adjoint-pairing-identity}
        F_{n,\lambda,W}(t)
        =
        \bigl\langle
                \phi_n^\infty(t),u^\infty
        \bigr\rangle
        =
        \bigl\langle
                \Rep_0^I\bigl(y_n,\Phi_n(t)\bigr),u
        \bigr\rangle .
\end{equation}

The network implements this identity in two passes.  The first pass computes the terminal scalar
\(H_n(t)\) while carrying the input \(t\) unchanged.  This scalar initializes a redundant atlas
state \(\Psi_0(t)\in E_I^*\oplus E_I^*\) that is safe-exact over \(y_0=[t]\), with witness
\(\Phi_0(t)\).  Starting again from \(y_0\), the second pass reruns the forward residual orbit and
propagates finite chart windows representing the intrinsic covectors
\(\Phi_0(t),\ldots,\Phi_n(t)\) in
\eqref{eq:rank-one-scaled-adjoint-cascade}.

\subsection{Continuous atlas initialization}
\label{ss:continuous-atlas-initialization}

Let
\(\Phi_0(t)=H_n(t)\Lambda_0(t)\in\mathcal E_{[t]}^*\)
be the intrinsic initial covector from
\eqref{eq:rank-one-initial-covector}.  
In the lift \(t\), the representative of
\(\Lambda_0(t)\) is the zero extension \(\lambda^\infty\), so the ordinary component is
\[
        \psi_{0,0}(t):=H_n(t)\lambda.
\]

The canonical half-shifted representative changes when \(t\) crosses the shifted-chart seam
\(s=\frac12\).  Away from the seam, its finite window is
\[
        H_n(t)(C_{-1}^I)^\trans\lambda
        \quad\text{for }t<s,
        \qquad
        H_n(t)\lambda
        \quad\text{for }t>s.
\]
The transition region is outside the safe set \(K_s\), so these two expressions may be joined
there by an arbitrary continuous CPwL interpolation.
Thus set
\[
        Z_-^{\rm init}(t)
        :=
        H_n(t)(C_{-1}^I)^\trans\lambda,
        \qquad
        Z_+^{\rm init}(t)
        :=
        H_n(t)\lambda.
\]
Choose a fixed CPwL function \(\xi_s:[0,1]\to[0,1]\) such that
\[
        \xi_s(t)=0
        \quad\text{for }0\le t\le s-\mar,
        \qquad
        \xi_s(t)=1
        \quad\text{for }s+\mar\le t\le1,
\]
and extend it globally as a CPwL function on \(\R\).  
Define
\begin{equation}
\label{eq:shifted-atlas-initialization}
        \psi_{s,0}(t)
        :=
        \Sw_{A_{\rm init}}
        \bigl(
                \xi_s(t);
                Z_+^{\rm init}(t),
                Z_-^{\rm init}(t)
        \bigr),
\end{equation}
with
\begin{equation}
\label{eq:atlas-initialization-scale}
        A_{\rm init}
        =
        \max\left\{
                1,\,
                \|h\|_\infty
                \bigl\|
                        (C_{-1}^I)^\trans\lambda-\lambda
                \bigr\|_\infty
        \right\} ,
\end{equation}
and set
\begin{equation}
\label{eq:atlas-initial-state}
        \Psi_0(t)
        :=
        \bigl(\psi_{0,0}(t),\psi_{s,0}(t)\bigr)
        \in E_I^*\oplus E_I^*.
\end{equation}

\begin{lemma}[Continuous atlas initialization]
\label{lem:continuous-atlas-initialization}
For every \(t\in[0,1]\), the state \(\Psi_0(t)\) defined by
\eqref{eq:shifted-atlas-initialization}--\eqref{eq:atlas-initial-state} is safe-exact over
\([t]\), with witness \(\Phi_0(t)\).  Equivalently,
\begin{equation}
\label{eq:initial-safe-exactness}
        [t]\in K_\beta
        \quad\Longrightarrow\quad
        \psi_{\beta,0}(t)
        =
        \Rep_\beta^I\bigl([t],\Phi_0(t)\bigr),
        \qquad
        \beta\in\mathcal A.
\end{equation}
On \(K_0\cap K_s\), the two components also satisfy the finite chart-transition identity.
\end{lemma}

\begin{proof}
Since \(|H_n(x)|\le\|h\|_\infty\), the two shifted-chart candidates satisfy
\[
        \|Z_-^{\rm init}(t)-Z_+^{\rm init}(t)\|_\infty
        \le
        \|h\|_\infty
        \bigl\|(C_{-1}^I)^\trans\lambda-\lambda\bigr\|_\infty
        \le A_{\rm init}.
\]
Thus the exact switch in \eqref{eq:shifted-atlas-initialization} applies throughout
\([0,1]\).

For \(0\le t<1\), the lift \(t\) is the chosen ordinary lift of \([t]\), and the corresponding
representative of \(\Phi_0(t)\) is \(H_n(t)\lambda^\infty\).  Hence
\[
        \Rep_0^I\bigl([t],\Phi_0(t)\bigr)
        =
        H_n(t)\lambda
        =
        \psi_{0,0}(t).
\]
At \(t=1\), one has \(y_n=[2^n]=[0]\), and therefore
\(H_n(1)=h([0])=0\).  Thus \(\Phi_0(1)=0\), and the same identity holds independently of the
chosen lift.  The ordinary component is therefore exact on all of \([0,1]\).

Suppose that \(0\le t\le s-\mar\).  Then \([t]\in K_s\), the chosen shifted lift is \(t+1\),
and \(\kappa_s([t])=1\).  The dual lift-change rule gives
\[
        \Rep_s^I\bigl([t],\Phi_0(t)\bigr)
        =
        p_I C_{-1}^*
        \bigl(H_n(t)\lambda^\infty\bigr)
        =
        H_n(t)(C_{-1}^I)^\trans\lambda
        =
        Z_-^{\rm init}(t).
\]
Since \(\xi_s(t)=0\), the switch returns \(Z_-^{\rm init}(t)\).

If \(s+\mar\le t<1\), then \([t]\in K_s\), the chosen shifted lift is \(t\), and
\(\kappa_s([t])=0\).  Consequently,
\[
        \Rep_s^I\bigl([t],\Phi_0(t)\bigr)
        =
        H_n(t)\lambda
        =
        Z_+^{\rm init}(t).
\]
Since \(\xi_s(t)=1\), the switch returns \(Z_+^{\rm init}(t)\).

At \(t=1\), the identity \(H_n(1)=0\) yields
\[
        Z_-^{\rm init}(1)
        =
        Z_+^{\rm init}(1)
        =
        0
        =
        \Rep_s^I\bigl([1],\Phi_0(1)\bigr),
\]
so the shifted component is exact there as well.  On the remaining interval
\(s-\mar<t<s+\mar\), one has \([t]\notin K_s\), and safe-exactness imposes no condition on the
shifted component.

This proves \eqref{eq:initial-safe-exactness}.  Finally, on \(K_0\cap K_s\), both components are
canonical finite chart windows of the same intrinsic covector \(\Phi_0(t)\); their compatibility
therefore follows from \Cref{lem:finite-chart-transition}.
\end{proof}

The initialization is a fixed CPwL function of \((t,H_n(t))\).  Thus, once the first pass has
computed and carried \(H_n(t)\) together with \(t\), forming \(\Psi_0(t)\) requires only a
constant additional amount of width and depth.
Set
\begin{equation}
\label{eq:reference-initial-state-bound}
        M_0
        :=
        \|h\|_{L^\infty([0,1])}
        \max\left\{
                \|\lambda\|_\infty,\,
                \|(C_{-1}^I)^\trans\lambda\|_\infty
        \right\}.
\end{equation}
The ordinary component is bounded by
\(\|h\|_\infty\|\lambda\|_\infty\), while the shifted component lies coordinatewise between its
two candidates.  Hence
\[
        \sup_{\tau\in[0,1]}
        \|\Psi_0(\tau)\|_\infty
        \le M_0.
\]
In particular, the bound \(M_0\) is independent of \(n\).

\subsection{Two-pass realization for reference atoms}
\label{ss:rank-one-two-pass-realization}

Fix \(n\ge1\).  In the first pass, the network carries \(t\) unchanged while advancing
\(z_0(t)=\emb([t])\) through \(n\) steps of the loop controller.  The scalar atlas readout from
\Cref{prop:two-grid-scalar-readout} then gives
\begin{equation}
\label{eq:reference-terminal-scalar}
        H_n(t):=H(z_n(t))=h([2^n t])=h(y_n).
\end{equation}
The initialization of \Cref{lem:continuous-atlas-initialization} uses the pair
\((t,H_n(t))\) to form a safe-exact state
\(\Psi_0(t)=(\psi_{0,0}(t),\psi_{s,0}(t))\), witnessed over \(y_0=[t]\) by
\(\Phi_0(t)=H_n(t)\Lambda_0(t)\).  Since the first pass has depth \(O(n)\) and the initialization
adds only constant overhead, \(\Psi_0\) is realized with fixed width and depth \(D_0(n)=O(n)\).
Moreover, its bound \(M_0\) from \eqref{eq:reference-initial-state-bound} is independent of \(n\).

The second pass restarts from \(z_0(t)\) and applies the atlas adjoint recursion for \(n\) steps.
By \Cref{thm:fixed-width-atlas-adjoint-recursion}, the resulting terminal state
\(\Psi_n(t)=(\psi_{0,n}(t),\psi_{s,n}(t))\) has fixed width and depth \(O(n)\).  Its scalar output
is the fixed linear readout
\begin{equation}
\label{eq:rank-one-network-output}
        \mathcal N_{n,\lambda,W}(t)
        :=
        \langle\psi_{0,n}(t),u\rangle.
\end{equation}
Note that the half-shifted component is needed only to propagate the atlas state across
intermediate chart seams.  At the terminal point, the reference profile is supported in the
ordinary safe region, so the ordinary component alone gives the required readout.

\begin{theorem}[Exact realization of reference rank-one contributions]
\label{thm:rank-one-realization}
Let \(W\) be the reference rank-one terminal atom determined by
\(h\in\CPwL([0,1])\) and \(u\in E_I\), and let \(\lambda\in E_I^*\).  For
\(t\in[0,1]\), let \(\Lambda_0(t)\in\mathcal E_{[t]}^*\) be the intrinsic covector represented by
the zero extension \(\lambda^\infty\) in the lift \(t\).
Then, for every \(t\in[0,1]\) and \(n\ge1\), the two-pass output satisfies
\begin{equation}
\label{eq:reference-atom-network-correctness}
        \mathcal N_{n,\lambda,W}(t)
        =
        F_{n,\lambda,W}(t)
        =
        \bigl\langle
                \Lambda_0(t),W_n([t])
        \bigr\rangle.
\end{equation}
Consequently, there exist integers \(C_0,C_1\ge1\), depending only on the fixed refinement and
atlas data and on \(h,u,\lambda\), such that
\begin{equation}
\label{eq:reference-network-class}
        F_{n,\lambda,W}
        \in
        \Ups_{C_0,C_1n}(\ReLU;1,1),
        \qquad
        n\ge1.
\end{equation}
\end{theorem}

\begin{proof}
The first pass computes
\(H_n(t)=h([2^n t])\)
with fixed width and depth \(O(n)\), while carrying \(t\) through an identity channel.  By
\Cref{lem:continuous-atlas-initialization}, the resulting map
\(t\mapsto\Psi_0(t)\) is safe-exact over \(y_0=[t]\), witnessed by
\(\Phi_0(t)=H_n(t)\Lambda_0(t)\).  The bounds
\eqref{eq:reference-initial-state-bound} is independent of \(n\), so
\Cref{thm:fixed-width-atlas-adjoint-recursion} applies.  It produces, with fixed width and
depth \(O(n)\), a terminal state
\(\Psi_n(t)=\bigl(\psi_{0,n}(t),\psi_{s,n}(t)\bigr)\)
that is safe-exact over \(y_n=[2^n t]\), witnessed by \(\Phi_n(t)\).

Suppose first that \(H_n(t)\ne0\).  Then the ordinary representative of \(y_n\) belongs to
\(\supp h\subset J_{\rm ref}\).  By
\eqref{eq:reference-interval-ordinary}, we therefore have \(y_n\in K_0\), and safe-exactness gives
\[
        \psi_{0,n}(t)
        =
        \Rep_0^I\bigl(y_n,\Phi_n(t)\bigr).
\]
The pairing identity \eqref{eq:rank-one-adjoint-pairing-identity} now yields
\[
        \mathcal N_{n,\lambda,W}(t)
        =
        \langle\psi_{0,n}(t),u\rangle
        =
        F_{n,\lambda,W}(t).
\]

If \(H_n(t)=0\), then the explicit initialization gives
\(\Psi_0(t)=(0,0)\), while \(\Phi_0(t)=0\).  The atlas update maps preserve the zero state, and
the intrinsic adjoint recursion is linear.  Hence
\[
        \Psi_n(t)=(0,0),
        \qquad
        \Phi_n(t)=0.
\]
Consequently, both \(\mathcal N_{n,\lambda,W}(t)\) and
\(F_{n,\lambda,W}(t)\) vanish.  This proves
\eqref{eq:reference-atom-network-correctness}.

The first pass and the atlas recursion each have fixed width and depth \(O(n)\).  The
initialization and terminal linear pairing add only fixed overhead.  Absorbing this overhead into
the constants gives \eqref{eq:reference-network-class}.
\end{proof}

Choosing \(\lambda\) among the standard coordinate covectors extracts the finitely many coordinate
patches needed for the subsequent gluing argument.

\begin{corollary}[Coordinate realization of a reference rank-one atom]
\label{cor:rank-one-coordinate-realization}
Let \(W\) be a reference rank-one terminal atom, and let
\(w_n^{(t)}\in E\) denote the representative of \(W_n([t])\) in the lift \(t\).
Then there exist integers \(C_0,C_1\ge1\), independent of \(n\) and \(i\in I\), such that
\[
        [t\mapsto (w_n^{(t)})_i]
        \in
        \Ups_{C_0,C_1n}(\ReLU;1,1),
        \qquad
        i\in I,\quad n\ge1.
\]

In particular, suppose that \(u=e_0\) and that \(W=W_h\) is the residual section induced by the
zero extension of \(h\).  Then
\begin{equation}
\label{eq:reference-function-patches}
        \bigl(w_n^{(t)}\bigr)_i
        =
        (V^nh)(t+i),
        \qquad
        i\in I.
\end{equation}
\end{corollary}

\begin{proof}
For each \(i\in I\), apply \Cref{thm:rank-one-realization} with
\(\lambda=e_i^*\).  Since \(\Lambda_0(t)\) is represented in the lift \(t\) by the zero extension
of \(e_i^*\), one has
\[
        \bigl\langle
                \Lambda_0(t),W_n([t])
        \bigr\rangle
        =
        \bigl(w_n^{(t)}\bigr)_i.
\]
The required constants can be chosen uniformly over \(i\) because \(I\) is finite.

For \(W=W_h\), the definition of the residual section \(W_f\) and the identity
\(W_m=W_{V^m f}\) following \eqref{eq:invariant-primal-cascade} give
\((w_n^{(t)})_i=(V^nh)(t+i)\), 
which proves \eqref{eq:reference-function-patches}.
\end{proof}

\section{General compactly supported CPwL seeds}
\label{sec:general-cpwl-seeds}

We now extend the realization result from the localized reference atoms of
\Cref{sec:rank-one-terminal-atoms} to arbitrary compactly supported CPwL seeds.  First, the
coordinate realizations of a reference atom are assembled across the relevant unit intervals to
obtain a realization of its refined iterate on \(\R\).  We then decompose a general seed into a
finite sum of translated reference atoms and use the linearity and translation covariance of the
homogeneous refinement operator.

\subsection{Global realization of reference unit atoms}
\label{ss:global-reference-unit-atoms}

Let \(h\in\CPwL(\R)\) satisfy \(\supp h\subset J_{\rm ref}\).
Let \(e_0\in E_I\) be the coordinate vector at index \(0\), and let \(W_h\) be the corresponding
intrinsic reference atom, whose representative in the ordinary lift \(\bar y\) is
\(h(\bar y)e_0^\infty\), where \(e_0^\infty:=\jmath_Ie_0\).
By \Cref{cor:rank-one-coordinate-realization}, for every \(i\in I\), the map
\(t\mapsto (V^nh)(t+i)\)
admits an exact fixed-width, depth-\(O(n)\) ReLU realization on \([0,1]\).

To assemble these unit-interval realizations, we use preservation of integer zeros.
Since \(\supp h\Subset(0,1)\), one has \(h(k)=0\) for every \(k\in\Z\).  More generally, if
\(g\) vanishes on \(\Z\), then, for every \(k\in\Z\), we have
\[
        (Vg)(k)
        =
        \sum_{j\in\Z} c_j g(2k-j)
        =
        0,
\]
because \(2k-j\in\Z\).  Hence \(V^nh\) vanishes at every integer for all \(n\ge0\).

\begin{proposition}[Global realization of reference unit atoms]
\label{prop:unit-endpoint-zero-atom}
Suppose that \(h\in\CPwL(\R)\) satisfies \(\supp h\subset J_{\rm ref}\).  
Then there exist integers \(C_0,C_1\ge1\), independent of \(n\), such that
\[
        V^nh\in\Ups_{C_0,C_1n}(\ReLU;1,1),
        \qquad
        n\ge1.
\]
Thus \(V^nh\) is realized exactly on \(\R\) by networks of fixed width and depth \(O(n)\).
\end{proposition}

\begin{proof}
The support-window hypothesis and
\(\supp h\subset[0,L]\) imply
\(\supp V^nh\subset[0,L]\) for every \(n\ge0\).  
For \(k=0,\ldots,L-1\), define
\[
        F_k^{(n)}(t)
        :=
        (V^nh)(t+k),
        \qquad
        0\le t\le1.
\]
Since \(k\in I_{\rm phys}\subset I\), each \(F_k^{(n)}\) has a fixed-width,
depth-\(O(n)\) realization by
\Cref{cor:rank-one-coordinate-realization}, uniformly over \(k\).

The patches satisfy the joining identities
\[
        F_k^{(n)}(1)
        =
        F_{k+1}^{(n)}(0),
        \qquad
        k=0,\ldots,L-2,
\]
because \(V^nh\) is continuous.  The integer-zero property also gives the outer endpoint
conditions
\[
        F_0^{(n)}(0)=0,
        \qquad
        F_{L-1}^{(n)}(1)=0.
\]
After relabeling \(f_{k+1}:=F_k^{(n)}\), the hypotheses of the gluing construction
\cite[Lemma~3.12]{loop} are satisfied.  It therefore combines the finitely many patches into a
global realization equal to \(V^nh\) on \([0,L]\) and to zero outside this interval.  The gluing
ramps evaluate each patch network only at inputs in \([0,1]\), where its realization is exact.
Since \(L\) is fixed, this step adds only fixed width and fixed depth.
\end{proof}

\subsection{Decomposition into translated reference atoms}
\label{ss:translated-reference-atoms}

We next show that every compactly supported CPwL seed is a finite linear combination of
translations of profiles supported in the fixed reference interval \(J_{\rm ref}\).
For \(\delta\in\R\), let
\[
        (\tau_\delta g)(t):=g(t-\delta).
\]

\begin{proposition}[Finite reference-atom decomposition]
\label{prop:reference-atomic-decomposition}
Every compactly supported CPwL function \(f:\R\to\R\) admits a representation
\begin{equation}
\label{eq:reference-atomic-decomposition}
        f
        =
        \sum_{\nu=1}^{N_*}\tau_{\delta_\nu}h_\nu,
        \qquad
        \supp h_\nu\subset J_{\rm ref},
\end{equation}
where \(N_*\), the translations \(\delta_\nu\), and the local CPwL complexities of the profiles
\(h_\nu\) are finite and independent of the refinement depth \(n\).
\end{proposition}

\begin{proof}
Let \(\ell_{\rm ref}:=|J_{\rm ref}|>0\).  Choose a compact interval \([a,b]\) containing
\(\supp f\) in its interior, and a finite partition
\[
        a=t_0<t_1<\cdots<t_N=b
\]
containing every breakpoint of \(f\).  Refining the partition if necessary, arrange that
\[
        t_{k+1}-t_k<\frac{\ell_{\rm ref}}2,
        \qquad
        k=0,\ldots,N-1.
\]
Since \(f(a)=f(b)=0\) and \(f\) is affine on each subinterval, its nodal expansion is
\[
        f
        =
        \sum_{k=1}^{N-1} f(t_k)\varphi_k,
\]
where \(\varphi_k\) is the nodal hat function at \(t_k\).  Its support is
\([t_{k-1},t_{k+1}]\), whose length is strictly smaller than \(\ell_{\rm ref}\).

For each \(k\) with \(f(t_k)\neq0\), choose \(\delta_k\in\R\) so that the translated support of
\(\varphi_k\) lies in \(J_{\rm ref}\), and define
\[
        h_k(t):=f(t_k)\varphi_k(t+\delta_k).
\]
Then \(\supp h_k\subset J_{\rm ref}\) and
\(\tau_{\delta_k}h_k=f(t_k)\varphi_k\).  Substituting these identities into the nodal expansion
and relabeling the nonzero terms gives \eqref{eq:reference-atomic-decomposition}.

All resulting data depend only on \(f\) and the fixed interval \(J_{\rm ref}\), and not on \(n\).
Moreover, each \(h_\nu\) is a translated nodal hat and therefore has uniformly bounded local CPwL
complexity.
\end{proof}

\subsection{Translation covariance and the main theorem}
\label{ss:translation-covariance-main-theorem}

The homogeneous binary refinement operator satisfies
\begin{equation}
\label{eq:translation-covariance}
        V^n(\tau_\delta g)
        =
        \tau_{2^{-n}\delta}(V^ng),
        \qquad
        \delta\in\R,\quad n\ge0.
\end{equation}
Indeed,
\(V(\tau_\delta g)(t)
=\sum_j c_j g(2t-j-\delta)
=(Vg)(t-\delta/2)
=\tau_{\delta/2}(Vg)(t)\),
and \eqref{eq:translation-covariance} follows by induction.

\begin{theorem}[Exact realization of binary refinement iterates]
\label{thm:main-realization}
Let
\[
        (Vf)(t)
        =
        \sum_{j\in\Z}c_jf(2t-j)
\]
be a scalar binary refinement operator with finitely supported mask, and assume that \(V\)
preserves the support window \([0,L]\) in the sense of
\eqref{eq:support-preservation}.  Let \(f:\R\to\R\) be compactly supported and CPwL, with
\(\supp f\subset[0,L]\).

Then there exist integers \(C_0,C_1\ge1\), depending only on the mask, \(L\), and a finite CPwL
description of \(f\), but not on \(n\), such that
\begin{equation}
\label{eq:main-realization-conclusion}
        V^nf
        \in
        \Ups_{C_0,C_1n}(\ReLU;1,1),
        \qquad
        n\ge1.
\end{equation}
\end{theorem}

\begin{proof}
By \Cref{prop:reference-atomic-decomposition}, write
\[
        f
        =
        \sum_{\nu=1}^{N_*}\tau_{\delta_\nu}h_\nu,
        \qquad
        \supp h_\nu\subset J_{\rm ref}.
\]
Linearity and \eqref{eq:translation-covariance} give
\[
        V^nf
        =
        \sum_{\nu=1}^{N_*}
        \tau_{2^{-n}\delta_\nu}(V^nh_\nu).
\]

For each \(\nu\), \Cref{prop:unit-endpoint-zero-atom} provides an exact fixed-width,
depth-\(O(n)\) realization of \(V^nh_\nu\).  Precomposing this realization with the affine map
\(t\mapsto t-2^{-n}\delta_\nu\) produces
\(\tau_{2^{-n}\delta_\nu}(V^nh_\nu)\) without changing its width or depth.  Although the affine
shift depends on \(n\), this affects only the network weights and biases.

Finally, the number \(N_*\) is independent of \(n\).  The finitely many translated realizations
can therefore be run in parallel and summed by an affine output layer, increasing the width only
by a fixed factor and leaving the depth \(O(n)\).  This proves
\eqref{eq:main-realization-conclusion}.
\end{proof}

\section{Conclusion}
\label{sec:conclusions}

We have given an exact fixed-width, depth-\(O(n)\) ReLU realization of finite homogeneous scalar
binary refinement iterates using a two-chart atlas selector.  The scalar binary conclusion was
already known by other methods; the contribution here is a different treatment of residual
discontinuities.  The loop controller of \cite{loop} transports the residual orbit exactly on a
polygonal model of the circle, while the residual cascade is formulated intrinsically as a local
system and represented by finite padded chart windows.

The load-bearing mechanism is atlas overlap agreement.  The ordinary and half-shifted charts have
disjoint same-chart branch cuts, and selector transitions are confined to trusted overlaps where the
competing local descriptions are simultaneously valid.  There, canonical windows of one intrinsic
covector satisfy exact finite transition identities, so the corresponding local adjoint updates agree
after passage to a common source chart.  The scalar atlas readout requires only the natural circle
compatibility condition \(h(0)=h(1)\), without positivity or endpoint vanishing.  Terminal
localization serves a separate purpose: reference rank-one atoms are supported inside one ordinary
safe chart, realized by a two-pass construction that reruns the forward residual orbit, and then
extended to general compactly supported CPwL seeds by translation covariance, finite atomic
decomposition, and finite patch gluing.

The one-dimensional binary case provides a complete model for broader multichart refinement
systems.  In such settings, residual discontinuities may be handled by exact agreement of local
atlas descriptions, while terminal data are reduced separately to local reference atoms.  The chart
geometry used here is fixed independently of the refinement depth, so no stage-dependent
sharpening of seam neighborhoods is required.  Any crude exponential gate scale that remains comes
from possible growth of the finite adjoint products rather than from the atlas selectors themselves.

\section*{Acknowledgments}

This work was supported by the Natural Sciences and Engineering Research Council of Canada
through its Discovery Grants program.

\end{document}